\documentclass[11pt, a4paper]{article}
\usepackage[utf8]{inputenc}
\usepackage[usenames,dvipsnames]{xcolor}
\usepackage{amsmath}
\usepackage{amsthm}
\usepackage{amsfonts}
\usepackage{amssymb}
\usepackage{array}  
\usepackage{graphicx} 
\usepackage{color}
\usepackage[english]{babel}
\usepackage{mathrsfs}
\usepackage{graphicx}
\usepackage{dsfont}
\definecolor{orangebis}{rgb}{0.99,0.25,0.00}
\definecolor{greenbis}{rgb}{0.10,0.85,0.10}
\definecolor{bluebis}{rgb}{0.10,0.30,0.99}
\usepackage[final]{hyperref}   
\usepackage{subcaption}

\usepackage{multicol}

\usepackage{lipsum}
\usepackage{geometry}
\author{Alejandro Rivera and Hugo Vanneuville}
\title{}
\date{}

\theoremstyle{plain}
\newtheorem{thm}{Theorem}[section]
\newtheorem{prop}[thm]{Proposition}
\newtheorem{lem}[thm]{Lemma}

\newtheorem{cor}[thm]{Corollary}
\newtheorem{claim}[thm]{Claim}

\newtheorem{lemma}[thm]{Lemma}

\theoremstyle{definition}
\newtheorem{defi}[thm]{Definition}

\newtheorem{notation}[thm]{Notation}
\newtheorem{condition}[thm]{Condition}
\newtheorem{remark}[thm]{Remark}

\theoremstyle{remark}

\newcommand{\N}{\mathbb{N}}
\newcommand{\R}{\mathbb{R}}
\newcommand{\C}{\mathbb{C}}

\newcommand{\Z}{\mathbb{Z}}

\newcommand{\Pro}{\mathbb{P}}
\newcommand{\prob}{\mathbb{P}}
\newcommand{\E}{\mathbb{E}}

\newcommand{\cross}{\text{\textup{Cross}}}
\newcommand{\cov}{\text{\textup{Cov}}}
\newcommand{\var}{\text{\textup{Var}}}
\newcommand{\piv}{\text{\textup{Piv}}}
\newcommand{\arm}{\text{\textup{Arm}}}
\newcommand{\Circ}{\text{\textup{Circ}}}

\newcommand{\Piv}{\text{\textup{Piv}}}

\newcommand{\vq}{\overrightarrow{q}}

\newcommand{\calQ}{\mathcal{Q}}

\newcommand{\card}{\text{\textup{Card}}}

\newcommand{\ar}{\text{\textup{Area}}}
\newcommand{\leng}{\text{\textup{Length}}}
\newcommand{\calX}{\mathcal{X}}

\renewcommand{\C}{\mathbb{C}}
\renewcommand{\Z}{\mathbb{Z}}
\renewcommand{\N}{\mathbb{N}}

\newcommand{\un}{\mathds{1}}

\newcommand{\cond}{\, \Big| \,}
\renewcommand{\textbf}[1]{\begingroup\bfseries\mathversion{bold}#1\endgroup}
\def\calA{\mathcal{A}}
\def\calB{\mathcal{B}}
\def\calC{\mathcal{C}}
\def\calD{\mathcal{D}}
\def\calE{\mathcal{E}}

\def\calH{\mathcal{H}}

\def\calK{\mathcal{K}}
\def\calL{\mathcal{L}}
\def\calN{\mathcal{N}}
\def\calQ{\mathcal{Q}}

\def\calS{\mathcal{S}}
\def\calT{\mathcal{T}}
\def\calU{\mathcal{U}}
\def\calV{\mathcal{V}}

\def\calX{\mathcal{X}}

\def\var{\mathop{\mathrm{Var}}}

\def\cov{\mathrm{Cov}}

\def\E{\mathbb{E}} 

\def \eps {\varepsilon}

\def\<#1{\langle #1\rangle}

\def\bi{\begin{itemize}}  
\def\ei{\end{itemize}}
\def\bnum{\begin{enumerate}} 
\def\enum{\end{enumerate}}
\def\ni{\noindent}
\def\bf{\bfseries}

\numberwithin{equation}{section}
\title{Quasi-independence for nodal lines}
\author{Alejandro Rivera\thanks{Univ. Grenoble Alpes, UMP5582, Institut Fourier, 38000 Grenoble, France, supported by the ERC grant Liko No 676999} \hspace{1cm} Hugo Vanneuville\thanks{Univ. Lyon 1, Institut Camille Jordan, 69100 Villeurbanne, France, supported by the ERC grant Liko No 676999}}
\date{}

\begin{document}

\maketitle

\abstract{We prove a quasi-independence result for level sets of a planar centered stationary Gaussian field with covariance $(x,y)\mapsto\kappa(x-y)$, with only mild conditions on the regularity of $\kappa$. As a first application, we study percolation for nodal lines in the spirit of~\cite{bg_16}. In the said article, Beffara and Gayet rely on Tassion's method (\cite{tassion2014crossing}) to prove that, under some assumptions on $\kappa$, most notably that $\kappa \geq 0$ and $\kappa(x)=O(|x|^{-325})$, the nodal set satisfies a box-crossing property. The decay exponent was then lowered to $16+\eps$ by Beliaev and Muirhead in \cite{bm_17}. In the present work we lower this exponent to $4+\eps$ thanks to a new approach towards quasi-independence for crossing events. This approach does not rely on quantitative discretization. Our quasi-independence result also applies to events counting nodal components and we obtain a lower concentration result for the density of nodal components around the Nazarov and Sodin constant from~\cite{nazarov2015asymptotic}.

\tableofcontents

\section{Introduction}\label{s.intro}


In this article, we prove a quasi-independence result for level lines of planar Gaussian fields and present two applications of this result. First, we use it to revisit and generalize the results by Gayet and Beffara~\cite{bg_16} who initiated the study of large scale connectivity properties for nodal lines and nodal domains of planar Gaussian fields. Second, we apply it to the study of the concentration of the number of nodal lines around the Nazarov and Sodin  constant (the constant $\nu$ of Theorem 1 of \cite{nazarov2015asymptotic}). Let $f$ be a planar centered Gaussian field. The \textbf{covariance function} of $f$ is the function $K:\R^2\times\R^2\rightarrow\R$ defined by:
\[
\forall x,y\in\R^2,\ K(x,y)=\E[f(x)f(y)] \, .
\]
We assume that $f$ is normalized so that for each $x\in\R^2$, $K(x,x)=\var(f(x))=1$, that it is non-degenerate (i.e. for any pairwise distinct $x_1,\dots,x_k\in\R^2$, $(f(x_1),\cdots,f(x_k))$ is non-degenerate), and that it is a.s. continuous and stationary. In particular, there exists a strictly positive definite continuous function $\kappa:\R^2\rightarrow[-1,1]$ such that $\kappa(0)=1$ and, for each $x,y\in\R^2$, $K(x,y)=\kappa(x-y)$. We will also refer to $\kappa$ as covariance function when there is no possible ambiguity. For each $p\in\R$ we call \textbf{level set} of $f$ the random set $\mathcal{N}_p:=f^{-1}(-p)$ and \textbf{excursion set} of $f$ the random set $\mathcal{D}_p:=f^{-1}([-p,+\infty[)$.\footnote{This convention, while it may seem counterintuitive, is convenient because it makes $\calD_p$ increasing both in $f$ and in $p$. See~\cite{rv_bf}.} Let us first state our result regarding planar box-crossing properties.

\paragraph{Box crossing estimates for planar Gaussian fields.} In~\cite{bg_16}, the authors give conditions under which such sets satisfy a box-crossing property at $p=0$. We say that random sets satisfy a box-crossing property if for any quad (i.e. a topological rectangle with two opposite distinguished sides) $\calQ$ there exists a positive constant $c$ such that for any (potentially sufficiently large) scale $s$, there is a crossing of $s\calQ$ between distinguished sides by the random set with probability larger than $c$. The study of the case $p=0$ is natural since this is the level at which duality arises, see for instance Remark~\ref{r.planar_duality} in our appendix. The most important conditions asked in~\cite{bg_16} were some symmetry conditions, the fact that $f$ is positively correlated (which means that the covariance function $\kappa$ takes only non-negative values) and a sufficiently fast decay for $\kappa(x)$ as $|x|$ does to $+\infty$, namely $\kappa(x)=O\left(|x|^{-325}\right)$. In~\cite{bm_17}, Beliaev and Muirhead have lowered the exponent $325$ to any $\alpha > 16$. In the present paper, we lower this exponent to any $\alpha > 4$, thus obtaining the following result:
\begin{thm}\label{t.RSW}
Assume that $f$ is a non-degenerate, centered, normalized, continuous, stationary, positively correlated planar Gaussian field that satisfies the symmetry assumption Condition~\ref{a.std} below. Assume also that $\kappa$ satisfies the differentiability assumption Condition~\ref{a.decay2} below and that $\kappa(x) \leq C |x|^{-\alpha}$ for some $C<+\infty$ and $\alpha > 4$. Let $\calQ$ be a quad, i.e. a simply connected bounded open subset of $\R^2$ whose boundary $\partial\calQ$ is piecewise smooth boundary with two distinguished disjoint segments on $\partial\calQ$. Then, there exists $c=c(\kappa,\calQ)>0$ such that for each $s \in ]0,+\infty[$, the probability that there is a continuous path in $\mathcal{D}_0\cap s\calQ$ joining one distinguished side to the other is at least $c$. Moreover, there exists $s_0<+\infty$ such that the same result holds for $\mathcal{N}_0$ as long as $s\geq s_0$.
\end{thm}

Lowering the exponent $\alpha$ below $4$, if at all possible, would require new ideas (see Remark~\ref{r.example}). This result is the analog of the Russo-Seymour-Welsh theorem for planar percolation from \cite{russo1978note,seymour1978percolation}, see also Lemma~4 of Chapter~3 of~\cite{bollobas2006percolation}, Theorem~11.70 and Equation~11.72 of~\cite{grimmett1999percolation} or Theorem~5.31 of~\cite{grimmett2010probability}. For more about the links between connectivity properties of nodal lines and domains and percolation, see~\cite{molchanov1983percolationi,molchanov1983percolationii,molchanov1986percolationiii}, \cite{alex_96}, \cite{bs_07}, \cite{bg_16}, \cite{bm_17}, \cite{bmw_17}, \cite{rv_bf}. Box-crossing estimates have previously been extended to some other dependent models, see~\cite{bollobas2006critical,duminil2011connection,tassion2014crossing, ahlberg2016sharpness} and also to some non-planar models, see~\cite{basu2015crossing,newman2017critical}. It seems also relevant to mention the recent work~\cite{bg_17}, in which the authors prove that the box-crossing property is stable by perturbations for sufficiently decorrelated discrete Gaussian fields. In particular, they obtain analogs of Theorem~\ref{t.RSW} for many discrete Gaussian fields that are not positively associated.\\

The result analogous to Theorem~\ref{t.RSW} in~\cite{bg_16} is Theorem~4.9. In~\cite{bm_17}, this is Theorem~1.7. Note that our assumptions about the differentiability and the non-degeneracy of $\kappa$ are different from those in~\cite{bg_16} and~\cite{bm_17}. Still, we see them essentially as technical conditions, whereas the question of the optimal exponent $\alpha$ seems to be of much more interest.\\

While our proof differs from the one in~\cite{bg_16,bm_17} in some key steps, the initial idea is the same, i.e. the use of Tassion's general method to prove box-crossing estimates which goes back to~\cite{tassion2014crossing}. Let us first be a little more precise about the proof in~\cite{bg_16,bm_17}. The three main ingredients are: i) a quantitative version of Tassion's method (see Section~2 of~\cite{bg_16}, ii) a quasi-independence result for finite dimensional Gaussian fields (see Theorem~4.3 of~\cite{bg_16} and Proposition~C.1 of~\cite{bm_17}) and iii) a quantitative approximation result (see Theorem~1.5 of~\cite{bg_16} and Theorems~1.3 and~1.5 of~\cite{bm_17}). Steps i) and ii) imply a discrete version of a RSW theorem and Step iii) is then used to deduce a RSW theorem for the continuous model. The most important contribution of~\cite{bm_17} is an improvement of the approximation result. Another way to prove the box-crossing property is to use prove a quasi-independence in the continuum and then apply Tassion's method (not necessarily in a quantitative way). This strategy was also suggested in \cite{bmw_17}, where Beliaev, Muirhead and Wigman prove a box-crossing estimate for random Gaussian fields on the sphere and the torus. More precisely, they used analogs of steps ii) and iii) above to prove such a quasi-independence result, see their Proposition~3.4. In the present work, we also prove a quasi-independence result in the continuum (see Theorem~\ref{t.qi_rect}) and then apply Tassion's method. However, the way we prove such a quasi-independence result is very different from~\cite{bmw_17}. In particular, \textit{we do not rely on any quantitative approximation result} and we rather prove a quasi-independence result \textit{uniform in the discretization mesh} (see Proposition~\ref{p.eps_qi}). Moreover, our techniques, together with the quantitative adaptation of~\cite{tassion2014crossing} presented in~\cite{bg_16} yield a uniform discrete RSW-estimate without any constraints on the mesh (see Proposition~\ref{p.eps_rsw}). This result is quite handy when using discrete techniques to study continuous fields, see for instance~\cite{rv_bf}. The proof of Theorem~\ref{t.RSW} is written in Section~\ref{s.Tassion} by relying only on our Sections~\ref{s.qi.vector} and~\ref{s.weak-qi} (but not on Subsection~\ref{ss.components}) and on~\cite{tassion2014crossing}. For other works relying on Tassion's method for box crossing estimates, see~\cite{ahlberg2016sharpness,duminil2016box}.\\

Before stating our quasi-independence results, let us state our result regarding the concentration of the number of nodal components of planar Gaussian fields.

\paragraph{A concentration from below around the Nazarov and Sodin constant for the number of nodal components.}

In~\cite{naso_nod}, Nazarov and Sodin prove that, if $g$ is a random spherical harmonic of degree $n$ on the $2$-dimensional sphere and if $N_0(n)$ is the number of nodal components (i.e. connected components of the $0$-level set) of $g$, then there exists a constant $c_{NS} \in ]0,+\infty[$ such that, for every $\eps > 0$, there exists $C=C(\eps)<+\infty$ and $c=c(\eps)>0$ such that for every $n \in \N$:
\begin{equation}\label{e.concentration_harmo}
\Pro \left[ \left| \frac{N_0(n)}{n^2}-c_{NS} \right| \geq \eps \right]  \leq C \exp(-c n ) \, .
\end{equation}
In other words, the number of nodal components divided by $n^2$ concentrates exponentially around a constant. In~\cite{nazarov2015asymptotic}, the same authors consider a much larger family of fields and obtain the much more general following result but without concentration.
\begin{thm}[Theorem~1 of~\cite{nazarov2015asymptotic}]\label{t.NS}
Assume that $f$ is a normalized, continuous, stationary planar Gaussian field which satisfies the spectral hypotheses Condition~\ref{a.spectral} below. Then, there exists a constant $c_{NS}=c_{NS}(\kappa) \in ]0,+\infty[$ such that, if $N_0(s)$ is the number of connected components of the nodal set $\calN_0$ contained in the box $[-s/2,s/2]^2$, then $N_0(s)/s^2$ goes to $c_{NS}$ as $s$ goes to $+\infty$ a.s. and in $L^1$.
\end{thm}

\begin{remark}
Their result is actually more general: they obtain a result for families of Gaussian fields on manifolds with translation-invariant local limits (see Subsection 1.2 of~\cite{nazarov2015asymptotic}).
\end{remark}

Theorem~\ref{t.NS} and the quasi-independence results of the present paper enable us to obtain a concentration result \textbf{from below} of $N_0(s)/s^2$ around $c_{NS}$:

\begin{thm}\label{t.concentration_NS}
Assume that $f$ is a normalized, continuous, stationary and non-degenerate planar Gaussian field which satisfies the spectral hypotheses Condition~\ref{a.spectral} below and the differentiability assumption Condition~\ref{a.decay2} below.With the same notations as Theorem~\ref{t.NS}, we have the following:
\bi 
\item[1.] if there exists $C<+\infty$ and $c>0$ such that for every $x \in \R^2$ we have $|\kappa(x)| \leq C \exp(-c|x|^2)$, then for every $\eps>0$ there exists $C_0=C_0(\kappa,\eps)<+\infty$ and $c_0=c_0(\kappa,\eps)$ such that for each $s\in\, ]0,+\infty[$:
\[
\Pro \left[ \frac{N_0(s)}{s^2} \leq c_{NS}-\eps \right] \leq C_0 \exp(-c_0 s) \, ;
\]
\item[2.] if there exists $C<+\infty$ and $\alpha > 4$ such that for every $x \in \R^2$ we have $|\kappa(x)| \leq C |x|^{-\alpha}$, then for every $\delta > 0$ and every $\eps > 0$, there exists $C_0=C_0(\kappa,\alpha,\delta,\eps)<+\infty$ such that for each $s\in\, ]0,+\infty[$:
\[
\Pro \left[ \frac{N_0(s)}{s^2}\leq c_{NS}-\eps \right] \leq C_0 s^{4-\alpha+\delta} \, .
\]
\ei
\end{thm}

An important example of a Gaussian field which satisfies the decorrelation hypothesis of Item~1 above is the \textbf{Bargmann-Fock field} which is the analytic Gaussian field $\, : \, \R^2 \rightarrow \R$ with covariance function $(x,y) \in (\R^2)^2 \mapsto \kappa(x-y)=\exp \left( -\frac{1}{2}|x-y|^2 \right)$. In some sense, this field is the local limit of the \textbf{Kostlan polynomials }which are random homogeneous polynomials on the sphere which arise naturally from real algebraic geometry, see for instance the introduction of~\cite{bg_16} or that of~\cite{bmw_17}. The analogue of Theorem~\ref{t.NS} is known for these polynomials (see~\cite{nazarov2015asymptotic}), but the concentration inequality~\eqref{e.concentration_harmo} is not known (neither from below nor from above). There are however two relevant results in this direction. The first, Corollary 1.10 of~\cite{let_16}, proves that the probability that there are no components in a prescribed region decays polynomially fast. The second, Theorem 1 of~\cite{gw_10}, deals with the other extreme and proves that polynomials of degree $d>>1$ whose number of nodal components is maximal up to a linear term in $d$ are exponentially rare in $d$. We hope that the proof of Theorem~\ref{t.concentration_NS} can be adapted in order to get the lower concentration part of~\eqref{e.concentration_harmo} with $n=\sqrt{d}$ for Kostlan polynomials of degree $d>>1$.

\begin{remark}
In~\cite{nazarov2015asymptotic}, the authors obtain Theorem~\ref{t.NS} in any dimension. We believe that our techniques could be extended to higher dimensions (probably with additional technicalities).
\end{remark}

\begin{remark}
As explained in the paragraph above about RSW results and as suggested in~\cite{bmw_17}, another way of obtaining quasi-independence results for nodal lines of planar Gaussian fields is to use the quasi-independence results for finite dimensional vectors and the quantitative discretization results, both from~\cite{bg_16,bm_17}. One could probably deduce Theorem~\ref{t.concentration_NS} from either~\cite{bg_16} or~\cite{bm_17}, though with slightly different assumptions, and more to the point, with a weaker Item~2 (more precisely, we believe that the exponent in the right hand side would be $16-\alpha+\delta$ instead.
\end{remark}

Before stating our quasi-independence results, we list the conditions on the Gaussian fields under which we work in this article.

\paragraph{Conditions on the planar Gaussian fields.}

We will assume that Condition~\ref{a.super-std} is true  in all the present paper. Then, Condition~\ref{a.std} will be useful to apply classical percolation arguments, Conditions~\ref{a.pol_decay} and~\ref{a.decay2} will be useful to obtain quasi-independence results, and finally Conditon~\ref{a.spectral} is the assumptions by Nazarov and Sodin to obtain their convergence result.

\begin{condition}\label{a.super-std}
The field $f$ is non-degenerate (i.e. for any pairwise distinct $x_1,\dots,x_k\in\R^2$, $(f(x_1),\cdots,f(x_k))$ is non-degenerate), centered, normalized, continuous, and stationary. In particular, there exists a strictly positive definite continuous function $\kappa \, : \, \R^2 \rightarrow [-1,1]$ such that $K(x,y) := \E \left[ f(x) f(y) \right] = \kappa(y-x)$ and $\kappa(0)=1$.
\end{condition}

\begin{condition}[Useful to apply percolation arguments.]\label{a.std}
The field $f$ is positively correlated, invariant by $\frac{\pi}{2}$-rotation, and reflection through the horizontal axis.
\end{condition}

\begin{condition}[Useful to have quasi-independence. Depends on a parameter $\alpha>0$.]\label{a.pol_decay}
There exists $C<+\infty$ such that for each $x\in\R^2$, $|\kappa(x)|\leq C|x|^{-\alpha}$.
\end{condition}

\begin{condition}[Technical conditions to have quasi-independence.]\label{a.decay2} The function $\kappa$ is $C^8$ and for each $\beta\in\N^2$ with $\beta_1+\beta_2\leq 2$, $\lim_{x\rightarrow \infty}\partial^\beta\kappa(x)=0$.
\end{condition}

\begin{condition}[Condition from~\cite{nazarov2015asymptotic}]\label{a.spectral}
Let $\rho$ be the spectral measure of $f$ which exists by Bochner's theorem (see~\cite{nazarov2015asymptotic}). Then: i) $\int_{\R^2} |\lambda|^4 \rho(d\lambda) < +\infty$, ii) $\rho$ has no atom, iii) $\rho$ is not supported on a linear hyperplane and iv) there exists a compactly supported signed measure $\mu$ whose support is included in the support of $\rho$ and a bounded domain $D \subseteq \R^2$ such that $\mathcal{F}(\mu)$ (the Fourier transform of $\mu$) restricted to $\partial D$ is non-positive and there exists $u_0 \in D$ such that $\mathcal{F}(\mu)(u_0)>0$.
\end{condition}

Note that, in the case of the Bargmann-Fock field, the spectral measure is simply a standard Gaussian measure, so this field satisfies Condition~\ref{a.spectral} (for the case iv), see Appendix~C of~\cite{nazarov2015asymptotic}). Moreover, $f$ is not degenerate since the Fourier transform of a continuous and integrable function $: \, \R^2 \rightarrow \R_+$ which
is not $0$ is strictly positive definite, see for instance Theorem~3 of Chapter~13 of~\cite{cheney2009course} (which is the strictly positive definite version of the easy part of Bochner theorem). Finally, the Bargmann-Fock field satisfies all the conditions above (and for every $\alpha > 0$).

\paragraph{The quasi-independence result.} Theorem~\ref{t.qi_rect} below is our quasi-independence result for level lines of planar Gaussian fields. We first need a few more notations. Consider the following setup: let $k_1,k_2 \in \Z_{>0}$ and let $(\calE_i)_{1\leq i\leq k_1+k_2}$ be a collections of either rectangles of the from $[a,b] \times [c,d]$ for some $a \leq b$ and $c \leq d$ or annuli of the form $x+[-a,a]^2 \setminus ]-b,b[^2$ for some $x \in \R^2$ and $a \geq b$. We say that a rectangle is crossed from left to right above (resp. below) $-p$ if there is a continuous path in $\calD_p$ (resp. $\calD_p^c$) included in this rectangle that joins its left side to its right side. Of course, an analogous definition holds for top-bottom crossings. Moreover, we say that there is a circuit above (resp. below) $-p$ in an annulus if there is circuit included in $\calD_p$ (resp. $\calD_p^c$) included in this annulus that separates its inner boundary from its outer boundary. Furthermore, for each $i \in \{ 1, \cdots, k_1+k_2 \}$, we let $N_p(i)$ denote the number of connected components of the level set $\calN_p$ which are included in $\calE_i$. We write $\calK_1=\cup_{i=1}^{k_1}\calE_i$, $\calC_1=\cup_{i=1}^{k_2}\partial\calE_i$, $\calK_2=\cup_{j=k_1+1}^{k_1+k_2}\calE_j$, and $\calC_2=\cup_{j=k_1+1}^{k_1+k_2}\partial\calE_j$.
\begin{thm}\label{t.qi_rect}
Let $f$ be a Gaussian field satisfying Conditions~\ref{a.super-std} and \ref{a.decay2} and consider the above setup. There exist $d=d(\kappa)<+\infty$ and $C=C(\kappa)<+\infty$ such that we have the following: let $p\in\R$. Let $A$ (resp. $B$) be an event in the $\sigma$-algebra generated by the crossings above $-p$ and below $-p$ of rectangles among the $(\calE_i)_{1\leq i\leq k_1}$ (resp. $(\calE_j)_{k_1+1\leq j\leq k_1+k_2}$), the circuits above $-p$ and below $-p$ in annuli among the $(\calE_i)_{1\leq i\leq k_1}$ (resp. $(\calE_j)_{k_1+1\leq j\leq k_1+k_2}$) and the variables $N_p(i)$ for $i \in \{1, \cdots,k_1 \}$ (resp. $i \in \{k_1+1, \cdots,k_1+k_2\}$). Let $\eta=\sup_{x\in\calK_1,y\in\calK_2}|\kappa(x-y)|$. If $\calK_1$ and $\calK_2$ are at distance greater than $d$, then:
\[
\left|\prob[A\cap B]-\prob[A]\prob[B]\right|\\
\leq\frac{C \,\eta}{\sqrt{1-\eta^2}}(1+|p|)^4 \, e^{-p^2} \prod_{i=1}^2\left(\ar(\calK_i)+\leng(\calC_i)+k_i\right)\, .
\]
\end{thm}

Note that in Theorem~\ref{t.qi_rect} we can consider crossing of rectangles (and similarly circuit in annuli) by \textbf{level} lines. Indeed, by Remark~\ref{r.planar_duality}, given a rectangle and for each $p \in \R$, a.s. there is a crossing of a rectangle included in $\calN_p$ if and only if there is such a crossing above $-p$ and a crossing below $-p$. The proof of Theorem~\ref{t.qi_rect} follows a perturbative technique applied to a discrete approximaion of our model (see Section~\ref{s.qi.vector}). To quantify the perturbation we control certain ``pivotal'' events using geometric techniques and the Kac-Rice formula (see Section~\ref{s.weak-qi}).

\begin{remark}\label{r.example}
If the perimeter of each of the rectangles and annuli of Theorem~\ref{t.qi_rect} is at most $s$, if $\calK_1$ and $\calK_2$ are at distance more than $s$ and if $\kappa(x) = O \left(|x|^{-\alpha} \right)$ then the right-hand-side of the estimates of Theorem~\ref{t.qi_rect} is:
\[
O \left( s^{4-\alpha} \left( 1+\frac{k_1+k_2}{s}+\frac{k_1k_2}{s^2} \right) \right) = O \left( k_1 k_2 s^{4-\alpha} \right) \, ,
\]
uniformly in $p$ as $s\rightarrow +\infty$ with $k_1$ and $k_2 $ fixed. Here we see how our condition $\alpha > 4$ from Theorems~\ref{t.RSW} and~\ref{t.concentration_NS} appears: $4$ equals $2$ times the dimension. It seems that it would require new ideas to cross this value.
\end{remark}

\begin{remark}
After the elaboration of this manuscript, the following works were brought to our attention:
\begin{itemize}
\item Piterbarg's mixing inequality (see for instance Theorem 1.2 of \cite{piterbarg}). This inequality is a more general version of our Proposition \ref{p.general_formula} below. We have chosen to keep it in the main body of the proof because we interpret and present it with a different point of view. See also Remark \ref{r.disclaimer}.
\item An almost independence result from~\cite{nsv_2007,nsv_2008,ns_2010}. In Theorem~3.1 of~\cite{ns_2010} (see also Theorem~3.2 of~\cite{nsv_2007} and Lemma~5 of~\cite{nsv_2008}), the authors derive a quasi-independence result for Gaussian entire functions. The result states roughly that a Gaussian entired function $f$, when restricted to a disjoint union of compact subsets of $\C$ not too large and far enough from each other, can be realized as a sum of independent copies of itself on each compact subset and a small perturbation. While the result is proved only for Gaussian entire functions, we believe it could apply to general Gaussian fields with sufficient decorrelation and regularity properties. To deduce a result similar to our Theorem~\ref{t.qi_rect} from Theorem~3.1 of~\cite{ns_2010}, one would need to understand how a perturbation of the field affects the events that we consider.
\end{itemize}
\end{remark}

\begin{remark}
At least one of the terms $\textup{Length}(\calC_i)$ and $k_i$ on the right-hand-side of the inequality in Theorem~\ref{t.qi_rect} must be present for the inequality to hold. Indeed, in their absence, we would have a quasi-independence result uniform in the choice (and number) of rectangles involved in the events $A$ and $B$ as long as these rectangles stay within prescribed sets $\calK_1$ and $\calK_2$. Moreover the excursion set $\calD_p$ is measureable with respect to the $\sigma$-algebra generated by the crossings of rectangles. Hence, we would have obtained the following result: let $\calK_1, \calK_2$ be two open subsets of the plane far enough from each other, let $p \in \R$ and let $A$ (resp. $B$) be an event measurable with respect to the excursion set $\calD_p\cap\calK_1$ (resp. $\calD_p\cap\calK_2$). Also, let $\eta=\sup_{x\in\calK_1,y\in\calK_2}|\kappa(x-y)|$ and assume that $\eta \leq 1/2$, then:
\begin{equation}\label{e.contradiction}
\left|\prob[A\cap B]-\prob[A]\prob[B]\right|\\
\leq C' \, \eta \, \ar(\calK_1) \, \ar(\calK_2)  \, .
\end{equation}
But this cannot be true in full generality. Indeed, let $f$ be the Bargmann-Fock field\footnote{For more information concerning the Bargmann-Fock field, we refer the reader to \cite{bg_16}.} described above, that is, the analytic Gaussian field with covariance $K(x,y)=e^{-\frac{1}{2}|x-y|^2}$. Then it is easy to see that $f$ satisfies Conditions~\ref{a.super-std} and \ref{a.decay2} so Theorem~\ref{t.qi_rect} applies. For each $s \in ]0,+\infty[$, let $A_s$ (resp. $B_s$) be the event that there is a continuous path in $\calN_0$ from $\partial [-s,s]^2$ (resp. $\partial [-3s,3s]^2$) to $\partial[-4s,4s]^2$. But $f$ is analytic and $\calN_0$ is a.s. smooth (see Lemma~\ref{l.transversality.1}) so $A_s$ is measureable\footnote{Indeed, a connected component of $\calN_0$ is a deterministic function of any segment of this component by unique analytic continuation and by the analytic implicit function theorem.} with respect to $\calD_0\cap [-2s,2s]^2$. On the other hand, $B_s$ is measureable with respect to $\calD_0\cap ([-4s,4s]^2\setminus]-3s,3s[^2)$. But $A_s$ implies $B_s$. Hence, if Equation~\eqref{e.contradiction} were valid, we would have
\[
O\left( s^4e^{-s^2/2}\right)=\left|\prob\left[A_s\cap B_s\right]-\prob\left[A_s\right]\prob\left[B_s\right]\right|=\prob\left[A_s\right]\prob\left[B_s^c\right]\, .
\]
But the Bargmann-Fock field satisfies the hypotheses of Theorem~\ref{t.RSW} so both $A_s$ and $B_s^c$ have probability bounded from below as $s\rightarrow +\infty$.
\end{remark}

\paragraph{Extension of the above results.} We believe that Theorem~\ref{t.qi_rect} above can be extended, in at least three directions. First, intead of considering rectangles and square annuli, one could consider quads (i.e. topological rectangles) and more general annuli. It seems that the treatment of the phenomena at the boundary will add new technical difficulties and we believe that, if we considered quads with piecewise smooth boundaries, then we might have obtained the same estimate as in Theorem~\ref{t.qi_rect} but with the following right hand side:
\[
\frac{C \,\eta}{\sqrt{1-\eta^2}}(1+|p|)^4 \, e^{-p^2} \prod_{i=1}^2\left(\ar(\calK_i)+\int_{\calC_i} (1+|\mathbf{k}|(t)) dt +k_i\right)\, ,
\]
where $dt$ is the length measure on the boundaries of the quads and $|\mathbf{k}|$ is the curvature (which is a Dirac mass at non-smooth points).\\

A second extension would be an extension to higher dimensions. We believe that the techniques of the present paper (except when we study the box-crossing property) are not restricted to the planar case. However, it seems that an extension to higher dimensions would add technical difficulties in intermediate lemmas of Section~\ref{s.weak-qi}.\\

A third extension would be to a larger class of events. It seems to be an interesting question to characterize a class of events for which our methods from Sections~\ref{s.qi.vector} and~\ref{s.weak-qi} work.

\paragraph{Proof Sketch.} The proof of Theorem~\ref{t.qi_rect} relies on an abstract quasi-independence result for threshold events of Gaussian vectors, namely Proposition \ref{p.general_formula}. In this proposition, given a Gaussian vector $X$ and two ``threshold events'' $\{X\in A\}$ and $\{X\in B\}$ measureable with respect to disjoint sets of coordinates (e.g. discrete crossing events of disjoint rectangles), we define a new Gaussian vector $Y$ whose covariance is close to that of $X$ such that $\{Y\in A\}$ and $\{Y\in B\}$ are independent. Next, we create a path $(X_t)_t$ of Gaussian vectors with $X_0=X$ and $X_1=Y$ and control the derivative of $\prob\left[X_t\in A\cap B\right]$ with respect to $t$ via ``pivotal'' events associated to $A$ and $B$. The path method we have just sketched is inspired by Slepian's proof of the normal comparison inequality (see Lemma 1 of~\cite{slepian_1962}). The only novelty so far is the interpretation of the quantities which arise as \textbf{probabilities of pivotal events}.\\
Once this core result is established, in Section~\ref{s.weak-qi}, we fix $A$ and $B$ as in\footnote{Actually, for simplicity, we begin with the case where $A$ and $B$ are crossing and circuit events. Once the proof is complete, we explain how to deal with the general case in Subsection~\ref{ss.components}.} the statement of Theorem~\ref{t.qi_rect}. Then, we discretize $\calK_1\cup\calK_2$ and approximate $A$ and $B$ by some discrete events $A^\eps$, $B^\eps$. We then prove the estimate of Theorem~\ref{t.qi_rect} for $A^\eps$ and $B^\eps$ with uniform bounds on $\eps$ and let $\eps$ go to $0$. This is the object of Proposition~\ref{p.eps_qi}. In order to prove the discrete inequality we first use Proposition~\ref{p.general_formula} for $X$ equal to $f$ restricted to the discretization, with $U=A^\eps$ and $V=B^\eps$. The right hand side is similar to the right hand side in Proposition~\ref{p.eps_qi}. The key is then to find good enough bounds for the probabilities of pivotal events. This is the object of Proposition~\ref{p.kac-rice}, at least for crossing events. The general case is dealt with in Subsection~\ref{ss.components}. Roughly speaking, if $x$ is an interior point, to be pivotal it must have four neighbors of alternating signs, so there is an $\eps$-approximate saddle point near $x$, which has probability $O(\eps^2)$. If $x$ is on the boundary (but not a corner), to be pivotal, it must have two neighbors with the same sign separated by a third neighbor with the opposite sign, all three on the same side of a line passing through $x$. We interpret this as a condition for the tangent of the nodal set at $x$ to belong to an angle of size $\eps$, which has probability $O(\eps)$. The proof of Proposition~\ref{p.kac-rice} is divided in two steps. The first is to show that pivotal events imply the existence of zeros of certain fixed derivatives of $f$. The arguments are of geometric nature and are presented in Subsection~\ref{ss.pivo_excep}. The second part is to prove that these events are indeed exceptional using Kac-Rice type arguments. This is done in Subsection \ref{ss.Kac_Rice}.

\paragraph{Outline.} In Section~\ref{s.qi.vector} we recall the key estimate needed to establish Theorem~\ref{t.qi_rect}, namely Proposition~\ref{p.general_formula}. We prove Theorem~\ref{t.qi_rect} (the quasi-independence thereorem for nodal lines) in Section~\ref{s.weak-qi}. More precisely, in Subsections~\ref{ss.qi_cross}, \ref{ss.pivo_excep} and \ref{ss.Kac_Rice} we prove this theorem in the case where $A$ and $B$ are generated by crossing events and then in Subsection~\ref{ss.components} we explain how to take into account the number of level lines components. In Section~\ref{s.Tassion}, we combine Theorem~\ref{t.qi_rect} (in the case of crossings) with Tassion's method (from~\cite{tassion2014crossing}) to obtain Theorem~\ref{t.RSW}. In Section~\ref{s.concentration_NS}, we use this theorem (in the case of number of nodal components) to obtain Theorem~\ref{t.concentration_NS} (concerning the lower concentration of the number of nodal components). Finally, in Appendix~\ref{s.basics} we recall classical results about Gaussian fields and in Appendix~\ref{s.eps.RSW} we prove a discrete box-crossing estimate uniform on the mesh, see Proposition~\ref{p.eps_rsw}.

\paragraph{Acknowledgements:}$ $\\
We are grateful to Christophe Garban and Damien Gayet for their helpful comments on the organization of the manuscript. We are also thankful for the stimulating discussions we had with them and with Vincent Beffara. Moreover, we would like to thank Stephen Muirhead for pointing to us the work by Piterbarg (after the elaboration of this manuscript). Finally, we wish to
thank the anonymous referee for his$\cdot$her helpful comments.

\section{Quasi-independence for Gaussian vectors}\label{s.qi.vector}
In this section, we reinterpret a classical quasi-independence formula of Gaussian vectors, namely Proposition~\ref{p.general_formula} below, which is at the heart of the proof of Theorem~\ref{t.qi_rect}. We first need to introduce some notation.

\begin{notation}
For any subset $U\subseteq\R^n$, write:
\[
\Piv_i(U)=\left\lbrace (x_1,\cdots,x_n)\in\R^n \, : \, \exists y_1,y_2\in\R, \,  \begin{array}{l}
(x_1,\cdots,x_{i-1},y_1,x_{i+1},\cdots,x_n)\in U,\\
(x_1,\cdots,x_{i-1},y_2,x_{i+1},\cdots,x_n)\notin U
\end{array}
\right\rbrace\, .
\]
\end{notation}

\begin{remark}
Note that $\Piv_i(U)$ is a subset of $\R^n$ that does not depend on the $i^{th}$ coordinate. Hence, we will sometimes see $\Piv_i(U)$ as a subset of $\R^{n-1}$ by forgetting the $i^{th}$ coordinate.
\end{remark}

\begin{remark}\label{r.piv}
For any $U,V\subseteq\R^n$ and any $i\in\{1,\dots,n\}$, we have:
\[
\piv_i(U)=\piv_i(U^c)\text{ and }\piv_i(U\cap V)\cup\piv_i(U\cup V)\subseteq\piv_i(U)\cup\piv_i(V)\, .
\]
\end{remark}

\begin{prop}\label{p.general_formula}
Let $k_1, k_2 \in \Z_{>0}$, let $X$ be a non-degenerate centered Gaussian vector of dimension $k_1+k_2$, and write $\Sigma$ for the covariance matrix of $X$. Assume that, for each $i \in \{ 1, \cdots, k_1+k_2 \}$, $\Sigma_{ii}=1$. Moreover, let $Y$ be a centered Gaussian vector of dimension $k_1+k_2$ independent of $X$ such that $(Y_i)_{1\leq i\leq k_1}$ has the same law as $(X_i)_{1\leq i \leq k_1}$, $(Y_j)_{k_1+1\leq j\leq k_1+k_2}$ has the same law as $(X_j)_{k_1+1\leq j \leq k_1+k_2}$, and the vectors $(Y_i)_{1\leq i\leq k_1}$ and  $(Y_j)_{k_1+1\leq j\leq k_1+k_2}$ are independent. For all $t\in [0,1]$, let $X_t=\sqrt{t}X+\sqrt{1-t}Y$. Furthermore, let $\vq \in \R^{k_1+k_2}$ let $U$ (resp. $V$) belong to the sub-$\sigma$-algebra of $\calB(\R^{k_1+k_2})$ generated by the sets $\lbrace x_i\geq q_i\rbrace$ for any $i\in\lbrace 1,\cdots,k_1\rbrace$ (resp. $i\in\lbrace k_1+1,\cdots,k_1+k_2\rbrace$). Then, we have:
\begin{multline*}
\big| \Pro \left[ X \in U\cap V \right]-\Pro \left[X \in U \right] \Pro \left[ X \in V \right] \big|\\
\leq \sum_{i \in \{ 1, \cdots, k_1 \},\atop j \in \{ k_1+1, k_1+k_2 \}} |\Sigma_{ij}| \, \int_0^1 \Pro \left[ X_t \in \Piv_i(U) \cap \Piv_j(V) \cond X_t(i)=q_i, \, X_t(j)=q_j \right] \, dt\\
\times \frac{1}{2\pi\sqrt{1-\Sigma_{ij}^2}}\exp \left(-\frac{q_i^2+q_j^2}{2} \right)\, .
\end{multline*}
\end{prop}

\begin{remark}\label{r.disclaimer}
Proposition \ref{p.general_formula} is a reinterpretation of a classical quasi-independence formula for Gaussian vectors used in quantitative versions of Slepian's Lemma (see \cite{slepian_1962} and Chapter 1 of \cite{piterbarg}, especially Theorem 1.1). The proof presented here is very close to that of \cite{piterbarg} except that we work in a level of generality more adapted to our purposes and that we introduce the notion of pivotal events, which are central in the proof of Theorem \ref{t.qi_rect}. Later, we use this definition to show that these probabilities are small for discrete approximations of crossing events.
\end{remark}
\begin{remark}
The proof of Proposition \ref{p.general_formula} is an interpolation argument. The path $X_t$ defined in the statement is an interpolation between $X$ and $Y$. By construction of $Y$, $\prob[Y\in U\cap V]=\prob[X\in U]\prob[X\in V]$ so the left hand side of the inequality can be written as
\[
\int_0^1\frac{d}{dt}\prob[X_t\in U\cap V]dt
\]
if you admit that the probability is differentiable. Now the first order of variation of this probability should correspond to how likely the events $X_t\in U$ and $X_t\in V$ are to change when $X_t$ one perturbs one of the $X_{t,i}$ and $X_{t,j}$ jointly, by a bump that depends on the shift in the covariance, which here is $\Sigma_{ij}$ if $i\leq k_1<j$ and $0$ otherwise. But this is precisely what is expressed in the right hand side of the inequality.
\end{remark}

\begin{lemma}\label{l.threshold}
Fix $n \in \Z_{>0}$, $\vq \in \R^n$ and let $U$ belong to the sub-$\sigma$-algebra of $\calB(\R^n)$ generated by the sets $\lbrace x_i\geq q_i\rbrace$ for $i\in\lbrace 1,\dots,n\rbrace$. Also, let $\varphi$ be a function which belongs to the Schwartz space $\calS(\R^n)$. Then, for each  $i\in\lbrace 1,\dots, n\rbrace$, there exists a measurable function $\epsilon_i=\epsilon_i(\varphi,U):\R^{n-1}\rightarrow\lbrace -1,0,1\rbrace$ such that:
\[
\int_U \frac{\partial \varphi}{\partial x_i}(x)dx = \int_{\Piv_i(U)}\epsilon_i(x_1,\dots,x_{i-1},x_{i+1},\dots,x_n)\, \varphi(x_1,\dots,x_{i-1},q_i,x_{i+1},\dots,x_n)\prod_{j\neq i}dx_j \, .\]
\end{lemma}
\begin{proof}
For each $\tilde{x}=(x_1,\dots,x_{i-1},x_{i+1},\dots,x_n)\in\R^{n-1}$, let $U_i(\tilde{x})$ be the set of $y\in\R$ such that $(x_1,\dots,x_{i-1},y,x_{i+1},\dots,x_n)\in U$. By Fubini's theorem:
\[
\int_{U}\frac{\partial \varphi}{\partial x_i}(x) \, dx =\int_{\R^{n-1}}\int_{U_i(\tilde{x})}\frac{\partial\varphi}{\partial x_i}(x)\,dx_i \, d\tilde{x} \, .
\]
Now, note that, for each $\tilde{x}$, $U_i(\tilde{x})$ equals either $\emptyset$, $\R$, $]-\infty,q_i[$, or $[q_i,+\infty[$. Moreover, if $\tilde{x}\notin \Piv_i(U)$, then $U_i(\tilde{x})=\R\text{ or }\emptyset$. Let $\epsilon_i(\tilde{x})$ be $1$ if $U_i(\tilde{x})=]-\infty,q_i[$, $-1$ if $U_i(\tilde{x})=[q_i,+\infty[$, and $0$ otherwise. By the fundamental theorem of analysis:
  \begin{align*}
   \int_{\R^{n-1}}\int_{U_i(\tilde{x})}\frac{\partial \varphi}{\partial x_i}(x)\,dx_i \, d\tilde{x} &=\int_{\R^{n-1}}\epsilon_i(\tilde{x}) \, \varphi(x_1,\dots,x_{i-1},q_i,x_{i+1},\dots,x_n)\,d\tilde{x}\\
    &=\int_{\piv_i(U)}\epsilon_i(\tilde{x}) \, \varphi(x_1,\dots,x_{i-1},q_i,x_{i+1},\dots,x_n)\,d\tilde{x} \, .
 \end{align*}
Note that Fubini's theorem and the fundamental theorem of analysis can be applied since $\varphi\in\mathcal{S}(\R^n)$.
\end{proof} 

\begin{proof}[Proof of Proposition~\ref{p.general_formula}]
Note that we have:
\begin{eqnarray*}
\Pro \left[ X \in U \cap V \right] - \Pro \left[ X \in U \right] \Pro \left[ X \in V \right] & = & \Pro \left[ X \in U \cap V \right] - \Pro \left[ Y \in U \cap V \right]\\
& = & \Pro \left[ X_1 \in U \cap V \right] - \Pro \left[ X_0 \in U \cap V \right] \, .
\end{eqnarray*}
Hence, it is sufficient to prove that, for each $t \in [0,1]$, we have:
\begin{multline}
\left| \frac{d}{dt} \Pro \left[ X_t \in U \cap V \right] \right| \\
\leq \sum_{i \in \{ 1, \cdots, k_1 \},\atop j \in \{ k_1+1, k_1+k_2 \}} |\Sigma_{ij}| \, \Pro \left[ X_t \in \Piv_i(U) \cap \Piv_j(V) \cond X_t(i)=q_i, \, X_t(j)=q_j \right]\\
\hspace{2cm} \times \frac{1}{2\pi\sqrt{1-\Sigma_{ij}^2}}\exp \left( -\frac{q_i^2+q_j^2}{2} \right)\, .\label{e.d/dt_sufficient}
\end{multline}

Note that since $X$ and $Y$ are non-degenerate and independent, for every $t \in [0,1]$, $X_t$ is non-degenerate. Moreover, $X_t$ has covariance $\Sigma_t$ defined as follows: $\Sigma_{t,ij}=\Sigma_{ij}$ if either $1\leq i,j\leq k_1$ or $k_1+1\leq i,j\leq k_1+k_2$, and $\Sigma_{t,ij}=t \Sigma_{ij}$ otherwise. Let $\Gamma :S_n^{++}(\R)\times\R^n\rightarrow\R$ be\footnote{Here  $S_n^{++}(\R)$ is the set of positive definite symmetric matrices of size $n$, that we see as the corresponding open subset of $\R^{\frac{n(n+1)}{2}} = \{ (\Sigma_{i,j})_{1 \leq i \leq j \leq n} \}$.} the map that associates to a matrix $\Sigma \in S_n^{++}(\R)$ and a point $x \in\R^n$ the Gaussian density at $x$ of a centered gaussian vector of covariance $\Sigma$. The function $\Gamma$ is $C^\infty$ and, for every $1 \leq i < j \leq n$, we have:\footnote{This is a classical property of Gaussian densities which follows immediately by application of the Fourier transform, see for instance Equation (2.3) of \cite{azws}.}
\begin{equation}\label{e.heat}
  \frac{\partial\Gamma}{\partial \Sigma_{i,j}}=\frac{\partial^2\Gamma}{\partial x_i\partial x_j} \, .
\end{equation}
Hence, by using dominated convergence and the chain rule:
\begin{eqnarray}
\frac{d}{dt} \Pro \left[ X_t \in U \cap V \right] & = & \sum_{1 \leq i \leq j \leq k_1+k_2} \frac{d \Sigma_{t,ij}}{dt} \int_{U \cap V} \frac{\partial}{\partial \Sigma_{ij}} \Gamma(\Sigma_t,x) \, dx\nonumber \\
& = & \sum_{i \in \{ 1, \cdots, k_1 \},\atop j \in \{ k_1+1, k_1+k_2 \}} \Sigma_{ij} \int_{U \cap V} \frac{\partial^2}{\partial x_i \partial x_j} \Gamma(\Sigma_t,x) \, dx \text{ by~\ref{e.heat}} \, . \label{e.d/dt_1}
\end{eqnarray}
Since $U$ depends only on the first $k_1$ coordinates and $V$ depends only on the $k_2$ last coordianates, we can apply Lemma~\ref{l.threshold} first to $(U,i,\frac{\partial}{\partial x_j}\Gamma(\Sigma_t,\cdot))$ and then to $(V,j,\Gamma(\Sigma_t,\cdot))$. We obtain that:
\begin{align}
&\left| \int_{U\cap V}\frac{\partial\Gamma}{\partial x_i\partial x_j}(\Sigma_t,x)dx\right| \nonumber \\
&\leq \int_{\Piv_i(U)\cap \Piv_j(V)} \Gamma(\Sigma_t,x_1\dots,x_{i-1},q_i,x_{i+1},\dots,x_{j-1},q_j,x_{j+1},\dots,x_{k_1+k_2}) \prod_{l \in \lbrace 1, \cdots, k_1+k_2 \rbrace, \atop l \notin \{ i,j \}} dx_l \nonumber \\
&=\prob \left[X_t \in \Piv_i(U)\cap \Piv_j(V) \cond X_t(i)=q_i, \, X_t(j)=q_j \right] \gamma_t(i,j) \, , \label{e.d/dt_2}
\end{align} 
where $\gamma_t(i,j)$ is the density of $(X_t(i),X_t(j))$ at $(q_i,q_j)$. Note that:
\begin{equation}\label{e.d/dt_3}
\gamma_t(i,j) \leq \frac{1}{2\pi\sqrt{1-(t\Sigma_{ij})^2}}\exp \left( -\frac{q_i^2+q_j^2}{2(1-t|\Sigma_{ij}|)} \right) \leq  \frac{1}{2\pi\sqrt{1-\Sigma_{ij}^2}}\exp \left( -\frac{q_i^2+q_j^2}{2} \right) \,  .
\end{equation}
Here, in the first inequality, we used the fact that if $A$ is a positive definite symmetric matrix, for any vector $X$, $\langle X,AX\rangle\geq \min\textup{sp}(A)\|X\|^2$. If we combine~\eqref{e.d/dt_1},~\eqref{e.d/dt_2} and~\eqref{e.d/dt_3}, we obtain~\eqref{e.d/dt_sufficient} and we are done.
\end{proof}

\section{Quasi-independence for planar Gaussian fields: the proof of Theorem~\ref{t.qi_rect}}\label{s.weak-qi}

In this section, we prove Theorem~\ref{t.qi_rect}. The steps of the proof are the following: we discretize our model, we apply Proposition~\ref{p.general_formula} to the discrete model, and then we estimate the probability of pivotal events that appear in the proposition. We refer the reader to the introduction for a rough sketch of the proof. Let us now introduce the discretization procedure (by following~\cite{bg_16}).\\

We work with the face-centered square lattice (see Figure~\ref{f.face-centered}) that we denote by $\mathcal{T}$. We denote by $\calT^\eps$ this lattice scaled by a factor $\eps$ and we denote by $\calV^\eps$ the vertex set of $\calT^\eps$. Given a realization of our Gaussian field $f$, some $p \in \R$ and some $\eps > 0$, the signs of the values of $f+p$ on the sites of $\calT^\eps$ is a site percolation model on $\calT^\eps$. It induces a random coloring of the plane defined as follows: For each  $x \in \R^2$, if $x \in \calV^\eps$ and $f(x) \geq -p$ or if $x$ belongs to an edge of $\calT^\eps$ whose two extremities $y_1,y_2$ satisfy $f(y_1) \geq -p$ and $f(y_2) \geq -p$, then $x$ is colored black. Otherwise, $x$ is colored white. In other words, we study a \textbf{correlated site percolation model} on $\cal T^\eps$. We also need the following definition.

\begin{figure}[!h]
\begin{center}
\includegraphics[scale=0.45]{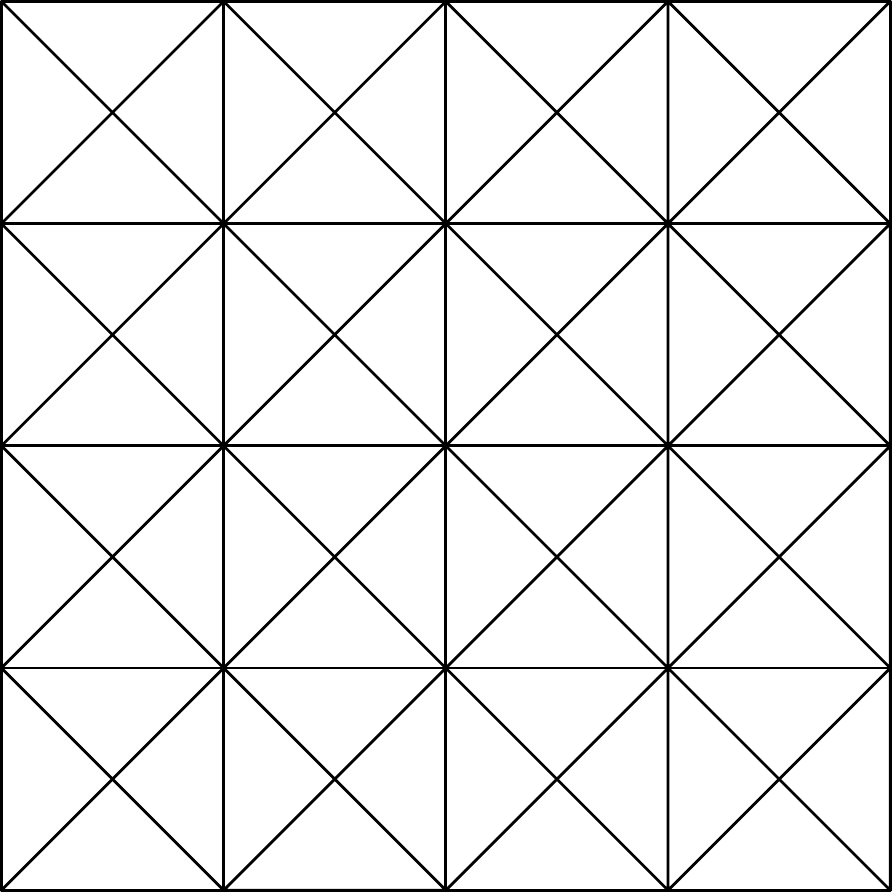}
\end{center}
\caption{The face-centered square lattice (the vertices are the points of $\Z^2$ and the centers of the squares of the $\Z^2$-lattice).}\label{f.face-centered}
\end{figure}

\begin{defi}\label{d.eps_drawn}
Given $\eps > 0$, an $\eps$-drawn rectangle is a rectangle of the form $[a,b]\times [c,d]$ where $a \leq b$ and $c \leq d$ are four integer multiples of $\eps$. An integer annulus is an annulus of the form $x+[-a,a]^2 \setminus ]-b,b[^2$ where $x \in (\eps\Z)^2$ and $a \leq b$ are two positive integer multiples of $\eps$.
\end{defi}

The specific choice of the face-centered square lattice is not very important. We will essentially use the following facts: i) $\calT$ is a triangulation, so we have nice duality arguments, see Remark~\ref{r.eps_duality} below, ii) $\calT$ is translation invariant,  iii) any $\eps$-drawn rectangle and any $\eps$-annulus can be drawn by using the edges of $\calT$, and iv) $\calT$ has nice symmetry properties. Actually, we will use the point iv) only in Section~\ref{s.eps.RSW}, but the results of this latter section are not used in the rest of the paper.\\

We start the proof of Theorem~\ref{t.qi_rect} by showing the result in the case where $A$ and $B$ are generated by crossing and circuit events since the proof is a little less technical in this case. This first part of proof is written in Subsections~\ref{ss.qi_cross},~\ref{ss.pivo_excep} and~\ref{ss.Kac_Rice}. Note that this partial result is already sufficient to prove Theorem~\ref{t.RSW}. We complete the proof of Theorem~\ref{t.qi_rect} by considering also the number of level lines components in Subsection~\ref{ss.components}.

\subsection{The proof of Theorem~\ref{t.qi_rect} in the case of crossing and circuit events}\label{ss.qi_cross}
In this subsection, we work only in the case of crossing and circuit events, we state Proposition~\ref{p.eps_qi}, a discrete analog of Theorem~\ref{t.qi_rect} with constants \textit{uniform in the mesh $\eps$}, and we deduce Theorem~\ref{t.qi_rect} (in the case of crossing and circuit events) from Proposition~\ref{p.eps_qi}. The proof of Proposition~\ref{p.eps_qi} is written in Subsections~\ref{ss.pivo_excep} and~\ref{ss.Kac_Rice}. Before stating this proposition, we need a definition:

\begin{defi}\label{d.crosseps}
Let $\eps>0$, $p\in\R$, and consider the above discrete percolation model. Also, let $\calE$ be a rectangle and $\calA$ be an annulus. We say that there is a left-right $\eps$-crossing of $\calE$ above (resp. below) $-p$ if there is a continuous black (resp. white) path included in $\calE$ from the left side of $\calE$ to its right side. We define top-bottom $\eps$-crossings similarly. We say that there is an $\eps$-circuit in $\calA$ above (resp. below) $-p$ if there is a continuous black (resp. white) path separating the inner boundary of $\calA$ from its outer boundary.
\end{defi}

\begin{remark}\label{r.eps_duality}
We will use the following duality argument which follows from the fact that $\calT$ is a triangulation and that any $\eps$-drawn rectangle and any $\eps$-drawn annulus can be drawn by using edges of $\calT^\eps$ (see Definition~\ref{d.eps_drawn}). Let $\eps > 0$, let $\calE$ be an $\eps$-drawn rectangle. Then, there is left-right crossing of $\calE$ above level $p$ if and only if there is no top-bottom crossing of $\calE$ below level $p$.
\end{remark}

\begin{prop}\label{p.eps_qi}
Let $f$ be a Gaussian field satisfying Conditions~\ref{a.super-std} and \ref{a.decay2}. There exists $d=d(\kappa)<+\infty$ and $C=C(\kappa)<+\infty$ such that we have the following: Let $p\in\R$ and $\eps \in ]0,1]$. Also, let $k_1,k_2 \in \Z_{>0}$ and let $(\calE_i)_{1\leq i\leq k_1+k_2}$ be a collections of either $\eps$-drawn rectangles or $\eps$-drawn annuli. Let
\[
\calK_1=\cup_{i=1}^{k_1}\calE_i,\ \calC_1=\cup_{i=1}^{k_2}\partial\calE_i,\ \calK_2=\cup_{j=k_1+1}^{k_1+k_2}\calE_j,\ \calC_2=\cup_{j=k_1+1}^{k_1+k_2}\partial\calE_j\, .
\]
Let $A^\eps$ (resp. $B^\eps$) be an event in the Boolean algebra generated by the left-right and top-bottom $\eps$-crossings above $-p$ and below $-p$ of rectangles among the $(\calE_i)_i$ for $1\leq i \leq k_1$ (resp. $(\calE_j)_j$ for $k_1+1\leq j \leq k_1+k_2$) and the $\eps$-circuits above $-p$ and below $-p$ in annuli among the $(\calE_i)_i$ for $1\leq i \leq k_1$ (resp. $(\calE_j)_j$ for $k_1+1\leq j \leq k_1+k_2$). Let $\eta=\sup_{x\in\calK_1,y\in\calK_2}|\kappa(x-y)|$. If $\calK_1$ and $\calK_2$ are at distance greater than $d$, then:
\[
\left|\prob[A^\eps\cap B^\eps]-\prob[A^\eps]\prob[B^\eps]\right|\leq\frac{C \,\eta}{\sqrt{1-\eta^2}}(1+|p|)^4 \, e^{-p^2} \prod_{i=1}^2\left(\ar(\calK_i)+\leng(\calC_i)+k_i\right)\, .
\]
\end{prop}

Note that the constant $C$ in Proposition~\ref{p.eps_qi} does not depend on $\eps$. Let us first show how Theorem~\ref{t.qi_rect} follows from Proposition~\ref{p.eps_qi} \textbf{in the case where the events $A$ and $B$ are generated by crossing and circuit events}. Also, \textbf{here and in all the rest of Section~\ref{s.weak-qi}, we assume that each of the $\calE_i$'s are rectangles}. The proof adapts easily to the case where the $\mathcal{E}_i$'s can also be annuli, but would be tedious to spell out.

\begin{proof}[Proof of Theorem~\ref{t.qi_rect}: Part 1 of 2, The case of crossings]
We assume that the events $A$ and $B$ are generated by crossing and circuit events. Also, we assume that each $\calE_i$ is a rectangle since the proof with annuli is exactly the same. First of all, using Lemma~\ref{l.transversality.1} and reasoning by approximation\footnote{Indeed, Lemma \ref{l.transversality.1} implies that crossing events for a given rectangle can be approximated by crossing events for approximations of this rectangle. Since $A$ and $B$ are generated by a finite boolean algebra of crossings, they can be obtained by a finite number of intersections and unions of crossings. Approximating each crossing and applying the same operations thus yields an approximation of $A$ and $B$.}, it is enough to prove the result for rectangles whose sides are integer multiples of some fixed $\eta>0$. But this is a direct consequence of Proposition~\ref{p.eps_qi} with $\eps_k=\eta/k$, with the same family of rectangles, and by taking the limit as $k$ goes to $+\infty$. Indeed, using Lemma~\ref{l.transversality.1} once more, it is easy to show that, if there is a (left-right, say) crossing of a rectangle above (resp. below) $-p$ in the continuum then a.s. there exists (a random) $\delta > 0$ such that this crossing belongs to a tube of width $\delta$ included in $\calD_p$ (resp. $\calD_p^c$). Hence, such a crossing in the continuum implies the analogous crossing in the discrete as long as $\eps_k < \delta$ and $\un_{A \setminus A^{\eps_k}}$ (resp. $\un_{B \setminus B^{\eps_k}}$) converges a.s. to $0$ as $k\rightarrow +\infty$. If there is no left-right crossing of a rectangle above (resp. below) $-p$, then (by Remark~\ref{r.planar_duality}) a.s. there is a top-bottom crossing below (resp. above) $-p$ of this rectangle so $\un_{A^{\eps_k} \setminus A}$ (resp. $\un_{B ^{\eps_k}\setminus B}$) converges a.s. to $0$ as $k\rightarrow +\infty$. Thus, we have shown Theorem~\ref{t.qi_rect} in the case where $A$ and $B$ are generated by crossing (and circuit) events.
\end{proof}
To prove Proposition~\ref{p.eps_qi}, we are going to use Proposition~\ref{p.general_formula}. We first define a Gaussian vector $X_t^\eps$ for each $t \in [0,1]$ in the spirit of the Gaussian vector $X_t$ from Proposition~\ref{p.general_formula}. Since we will apply intermediate lemmas to the underlying continuous Gaussian fields, we first define a field $f_t$ for every $t \in [0,1]$ as follows:
\begin{notation}\label{n.f_t}
Let $f$, $(\calE_i)_{1\leq i\leq k_1+k_2}$, $\calK_1$, $\calK_2$, $\calC_1$ and $\calC_2$ be as in Proposition~\ref{p.eps_qi}. Let $\calU_1$ and $\calU_2$ be disjoint neighborhoods of $\calK_1$ and $\calK_2$ respectively. Let $g$ be a continuous Gaussian field indexed\footnote{The reason we extend $g$ to open neighborhoods of $\calK_1$ and $\calK_2$ is largely technical and can be ignored during first reading.} by $\calU_1 \cup \calU_2$ independent of $f$ such that $g$ restricted to either of the $\calU_i$'s has the same law as $f$ restricted to $\calU_i$ and such that $g$ restricted to $\calU_1$ is independent of $g$ restricted to $\calU_2$. For each $t\in[0,1]$, let $f_t = \sqrt{t}f+\sqrt{1-t}g$. Note that (since $f$ is centered and non-degenerate) for each $t\in[0,1]$, $f_t$ is a non-degenerate centered Gaussian field whose covariance function is:
\[
\begin{cases}
\E \left[ f_t(x) f_t(y) \right] = \kappa(x-y) & \text{ if } x,y \in \calU_1 \text{ or } x,y \in \calU_2 \, ,\\
\E \left[ f_t(x) f_t(y) \right] = t\kappa(x-y) & \text{ otherwise} \, .
\end{cases}
\]
Also, for each $i \in \{ 1,2 \}$, let $\calV_i^\eps = \calK_i \cap \calV^\eps$, and let $X^\eps$ (resp. $X_t^\eps$) be $f$ (resp. $f_t$)  restricted to $\calV^\eps_1 \cup \calV_2^\eps$.
\end{notation}

We need one last notation before beginning the proof:
\begin{notation}\label{n.UandV}
Given $\eps$, $p$, $(\calE_i)_{1\leq i\leq k_1+k_2}$, $A^\eps$ and $B^\eps$ as in Proposition~\ref{p.eps_qi}, we write $\calV^\eps_1$ and $\calV_2^\eps$ as in Notation~\ref{n.f_t} and we write $U^\eps$ and $V^\eps$ for the corresponding Borelian subsets of $\R^{\calV^\eps_1 \cup \calV^\eps_2}$ i.e. the elements of the Boolean algebra generated by the sets $\{ x_i \geq -p \}$ for any $i \in \calV^\eps_1 \cup \calV^\eps_2$ such that:
\[
A^\eps = \{ X^\eps \in U^\eps \} \text{ and } B^\eps = \{ X^\eps \in V^\eps \} \, .
\]
\end{notation}

Let us now start the proof of Proposition~\ref{p.eps_qi}. By applying Proposition~\ref{p.general_formula} to $X^\eps$ (which is centered, normalized and non-degenerate since $f$ is centered, normalized and non-degenerate), $U^\eps$ and $V^\eps$, it is sufficient to prove that there exists $C=C(\kappa)<+\infty$ and $d=d(\kappa)<+\infty$ such that, if $\calK_1$ and $\calK_2$ are at distance greater than $d$ then for each $t \in [0,1]$ we have:
\begin{multline}\label{e.sufficient}
\sum_{x\in\calV^\eps_1, \atop y\in\calV^\eps_2} \prob \left[ X_t^\eps \in \piv_x(U^\eps)\cap\piv_y(V^\eps) \cond f_t(x)=f_t(y)=-p\right]\\ \leq C \, (1+|p|)^4\, e^{-p^2} \, \prod_{i=1}^2\left(\ar(\calK_i)+\leng(\calC_i)+k_i\right) \, .
\end{multline}

To prove~\eqref{e.sufficient}, we need to find good enough bounds for the probabilities of pivotal events.
This is the purpose of Subsections~\ref{ss.pivo_excep} and~\ref{ss.Kac_Rice}. The proof sketch provided in the introduction can be a useful guide to read the following subsections. Remember also that we have assumed that all of the $\calE_i$'s are rectangles.

\subsection{Pivotal sites imply exceptional geometric events}\label{ss.pivo_excep}

In this subsection, we fix a point $x$ on the $\eps$-lattice and explain how the fact that $x$ is pivotal for the discretized event $U^\eps$ implies the cancellation of certain derivatives of the field. The results are combined in three lemmas that we state together before proving them for future reference. Each proof is independent from the rest.\\
In the first lemma, we show that, roughly speaking, on the neighbors of a pivotal point $x$, the field must have alternating signs relative to $p$.
\begin{lemma}\label{l.qi.piv_to_geometry}
We use the same notations as in Notation~\ref{n.UandV} (remember in particular that $\calK_1 = \cup_{i=1}^{k_1} \calE_i$ and $\calC_1 = \cup_{i=1}^{k_1} \partial \calE_i$). Let $x \in \calV_1^\eps$, let $\omega^\eps \in \Piv_x(U^\eps) \subseteq \R^{\calV^\eps_1 \cup \calV^\eps_2}$ and call black (resp. white) a vertex $y \in \calV^\eps_1 \cup \calV^\eps_2$ such that $\omega^\eps(y) \geq -p$ (resp. $\omega^\eps(y) < -p$). If the point $x$ belongs to $\calK_1 \setminus \calC_1$, then it has four neighbors $x_1,x_2,x_3,x_4$ in anti-clockwise order around $x$ and of alternating color. If the point $x$ belongs to $\calC_1$ and is the corner of none of the $\calE_i$'s, then $x$ has three neighbors $x_1,x_2,x_3$ in anti-clockwise order around $x$ belonging to a common half-plane bounded by a line through $x$ and of alternating color.
\end{lemma}

In the last two lemmas, we explain how the information obtained in Lemma~\ref{l.qi.piv_to_geometry} implies the cancellation of certain derivatives of the field on fixed segments. The arguments are entirely deterministic.

\begin{lemma}\label{l.boundary}
Consider $\varphi\in C^1(\R^2)$ and $x,x_1,x_2,x_3\in\R^2$. Assume that any two distinct vectors $x-x_i$ for $i=1,2,3$ do not point in the same direction and that the $x_i$ are numbered in anti-clockwise order around $x$. Assume that
\begin{itemize}
\item We have $\varphi(x)=0$, $\varphi(x_1),\varphi(x_3) \geq 0$ and $\varphi(x_2)\leq 0$.
\item There is a closed half plane $H$ such that $x\in\partial H$ and $x_1,x_2,x_3\in H$.
\end{itemize}
Then, there exists $i\in\{1,2,3\}$ such that if $l=[x,x_i]$ has tangent vector $v$, $\partial_v\varphi$ has a zero on $l$.
\end{lemma}

Lemma \ref{l.boundary} essentially states the following: If $x$ is a point on the boundary of the rectangle such that $\varphi(x)=0$ and such that, as one goes around $x$ along a small half circle inside the rectangle, one encounters alternating color, then, the tangent vector of the nodal line of $\varphi$ containing $x$ must take some specific values near $x$. We formalize this by saying that restrictions of $\varphi$ to certain small segments near $x$ must have critical points.
\begin{lemma}\label{l.interior}
Consider $\varphi\in C^1(\R^2)$ and $x,x_1,x_2,x_3,x_4\in\R^2$. Assume that two vectors $x-x_i$ $i=1,2,3,4$ do not point in the same direction and that the $x_i$'s are numbered in anti-clockwise order around $x$. Assume also that:
\[
\text{ We have } \varphi(x)=0, \, \varphi(x_1),\varphi(x_3)\geq 0 \text{ and } \varphi (x_2),\varphi(x_4)\leq 0 \, .
\]
Let $d_0$ denote the diameter of $\{ x,x_1,\cdots,x_4 \}$. Then, there exist a finite set $\mathfrak{V}$ of unit vectors and a constant $C_0<+\infty$ both depending only on the angles between the segments $[x,x_i]$'s such that the following holds: There exist two segments $l_1$ and $l_2$ with non-colinear unit tangent vectors $v_1,v_2 \in \mathfrak{V}$, of length at most $C_0 \, d_0$ and both passing through at least one of the points $x,x_1,\cdots,x_4$ such that $\partial_{v_1}\varphi$ has a zero on $l_1$ and $\partial_{v_2}\varphi$ has a zero on $l_2$.
\end{lemma}
Lemma \ref{l.interior} roughly says that if $\varphi$ changes signs four times when going around $x$ along a small circle, then it must have an approximate saddle point at $x$. We formalize the notion of approximate saddle point by saying that there are two non-colinear segments of length $\eps$ on which the function $\varphi$ has a vanishing derivative. In the proof we distinguish several cases depending  on the relative positions of the $x_i$'s and the gradient of $\varphi$ at $x$. This reduces the proof to a planar euclidean geometry problem.
\begin{proof}[Proof of Lemma~\ref{l.qi.piv_to_geometry}]
By Remark~\ref{r.piv} we may assume that there exists $i_0\in\{1,\dots,k_1\}$ such that $U^\eps$ is the Borelian subset of $\R^{\calV^\eps_1 \cup \calV^\eps_2}$ which corresponds to the left-right crossing of $\calE_{i_0}$. If $x \notin \calE_{i_0}$ then $\Piv_x(U^\eps)$ is empty. If $x \in \calE_{i_0} \setminus\partial\calE_{i_0}$ and $\omega^\eps \in \Piv_x(U^\eps)$, then there are two paths made of black vertices connecting $x$ to left and right sides of $\calE_{i_0}$ and two white paths made of white vertices connecting $x$ to the top and bottom sides of $\calE_{i_0}$. These paths are necessarily of alternating color around $x$, so in particular it has four neighbors of alternating color. This proves the first assertion. Let $x\in \calC_1 \cap \calE_{i_0}$ such that $x$ is not a corner. If $x\notin \partial\calE_{i_0}$ then, as before, $x$ must have four neighbors of alternating color. But then among these, there must be three neighbors belonging to the same half-space bounded by $x$ with the properties required by the second assertion. On the other hand, if $x\in\partial\calE_{i_0}$, then there must be a path of one color starting at a neighbor of $x$ and reaching the opposite side of the rectangle and two additional paths of the opposite color connecting neighbors of $x$ to each of the adjacent sides to the one containing $x$. But then, the three neighbors at which these paths start are in the configuration announced by the second assertion.
\end{proof}

\begin{proof}[Proof of Lemma \ref{l.boundary}]
See Figure~\ref{f.lines} (a) for a snapshot of the proof. If $\nabla \varphi(x)=0$ then the result is trivial so assume that $\nabla \varphi(x) \neq 0$. Then, this gradient separates the plane into two closed half-spaces $H_+$ and $H_-$ such that $x \in \partial H_+ = \partial H_-$,  $\nabla \varphi(x)$ is orthogonal to this boundary, and $\nabla \varphi(x)$ points toward $H_+$. We distinguish between two cases: i) There exists $i_0 \in \{ 1,3 \}$ such that $x_{i_0} \in H_-$. In this case, let $l=[x,x_{i_0}]$ with unit vector $v$. Then, $\partial_v \varphi (x)\leq 0$, $ \varphi (x)=0$ and $\varphi(x_{i_0}) \geq 0$. Therefore, $\partial_vf$ must vanish somewhere on $l$. ii) The point $x_2$ belongs to $H_+$ (which happens if the case i) does not hold by the existence of the half-plane $H$ and since the $x_i$'s are in anti-clockwise order around $x$). In this case, the same argument works with $l=[x,x_2]$.
\end{proof}

\begin{figure}[!h]
\begin{center}
\includegraphics[scale=0.36]{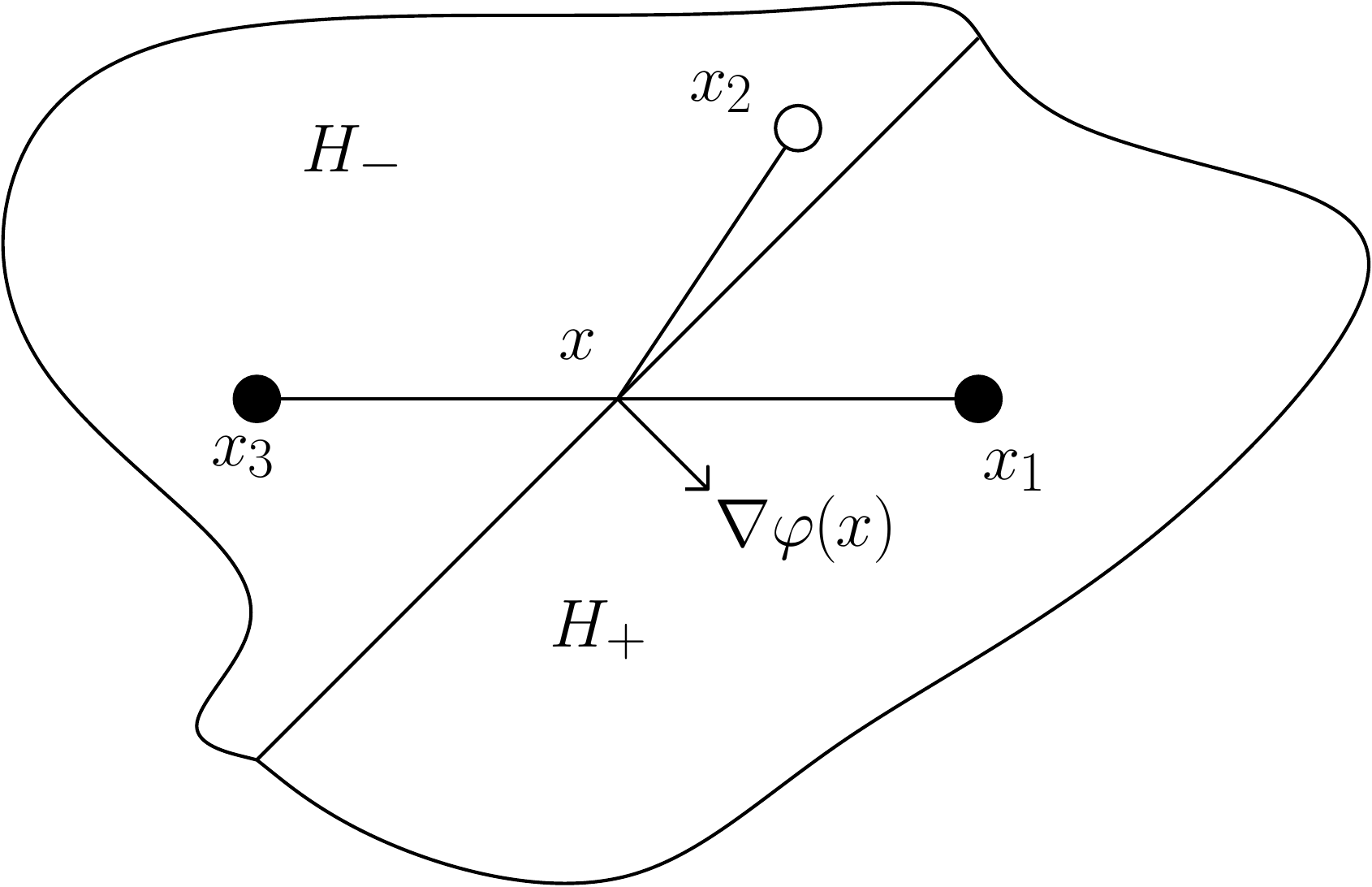}
\hspace{1em}
\includegraphics[scale=0.36]{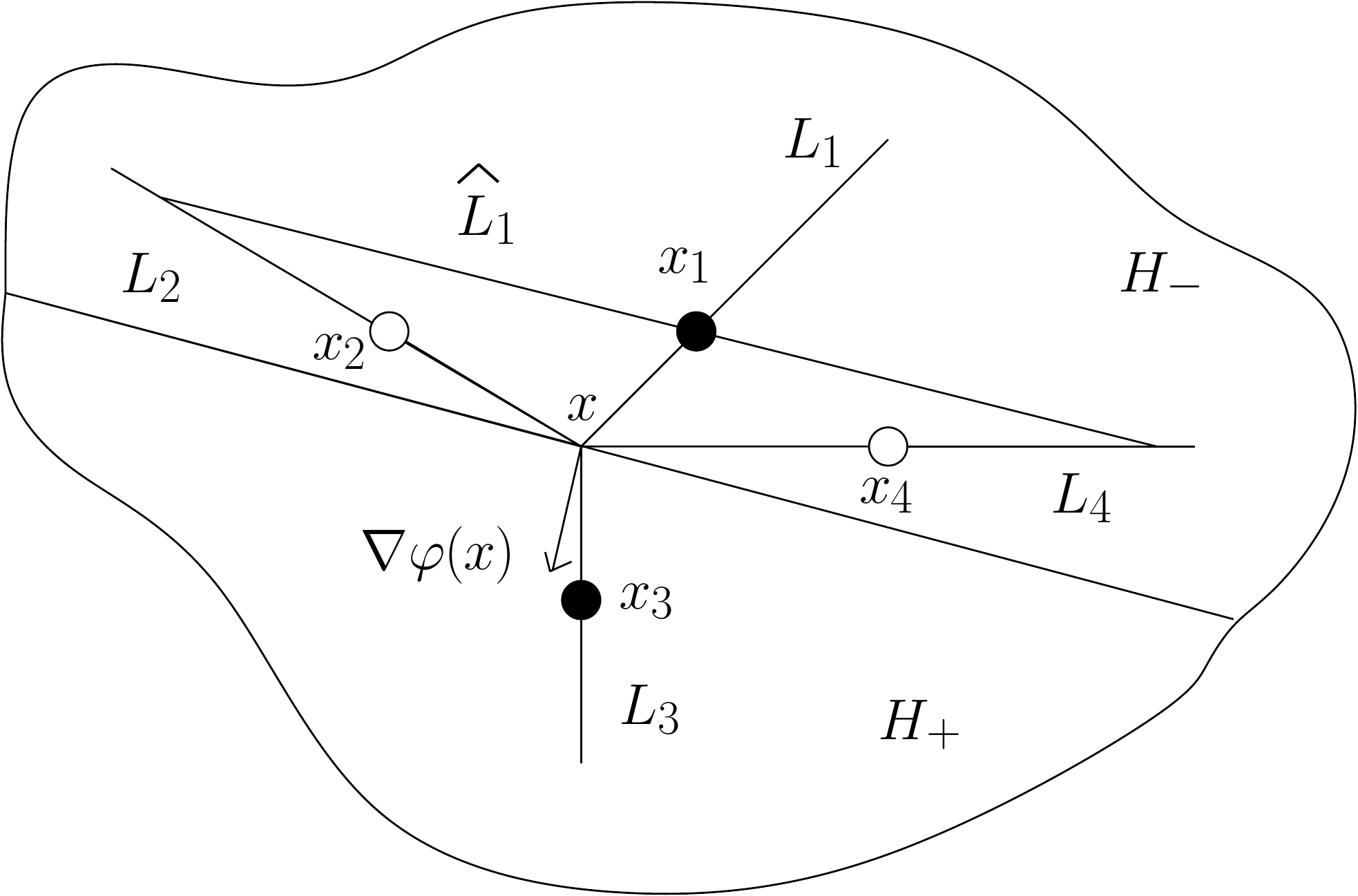}
\end{center}
\caption{(a) The proof of Lemma~\ref{l.boundary}, more particularly the case i) with $i_0=3$. (b) The proof of Lemma~\ref{l.interior}, more particularly the case ii).}\label{f.lines}
\end{figure}
\begin{proof}[Proof of Lemma~\ref{l.interior}]
See Figure~\ref{f.lines} (b) for an illustration of the proof. For each $i\in\{1,2,3,4\}$, let $L_i$ be the line $[x,x+C_0(x-x_i)]$ for some $C_0>0$ to be chosen later. If the anti-clockwise angle $\theta_i$ between $L_{i-1}$ and $L_{i+1}$ is less than $\pi$ (the indices should be read modulo $4$), set $\widetilde{L}_i:=[x_{i-1},x_{i+1}]$ and define $\widehat{L}_i$ to be the segment intersecting the bisector of $\theta_i$ orthogonally at $x_i$ and whose extremities belong to $L_{i-1}$ and $L_{i+1}$. We fix $C_0$ large enough so that whenever $\theta_i$ is indeed less than $\pi$, $L_i$ is long enough to intersect $\widetilde{L}_i$. We will choose $l_1$ and $l_2$ among the $L_i$'s, the $\widehat{L}_i$'s and the $\widetilde{L}_i$'s. The choice will follow by considering several distinct cases. In each case, the critical point will be detected either by finding three consecutive points on the segment on which $\varphi$ takes alternating signs, or by finding a point on the segment where $\varphi$ vanishes and has, say, a positive derivative, and proving that $\varphi$ takes a negative value further along the segment. In both cases, the existence of the critical point follows by Rolle's theorem.\\

As in the proof of Lemma~\ref{l.boundary}, note that if $\nabla \varphi(x)=0$ then the result is trivial so assume that $\nabla \varphi(x) \neq 0$. Then, this gradient separates the plane into two closed half-spaces $H_+$ and $H_-$ such that $x \in \partial H_+ = \partial H_-$,  $\nabla \varphi(x)$ is orthogonal to this boundary, and $\nabla \varphi(x)$ points toward $H_+$. Note that there are at least two consecutive points among the $x_i$'s in $H_-$ or two consecutive points in $H_+$, such that they do not both belong to $\partial H_- = \partial H_+$. Without loss of generality, assume that $x_1,x_2 \in H_-$ and that they do not belong both to $\partial H_-$. Then, along the segment $L_1$, $\varphi$ starts at $x$ with value $0$ and a non-positive derivative and $\varphi(x_1)\geq 0$. In particular, its derivative along this segment must vanish. We now distinguish between two cases:
\begin{itemize}
\item Assume that there exists $i\in\{2,3,4\}$ with $x_i\in H_-$ such that, first, $x_1$ and $x_i$ are not both on $\partial H_-$, and second, $f(x')\geq 0$ for some $x'\in L_i$. Then $\{ l_1,l_2 \} = \{ L_1,L_i \}$ satisfies the required conditions (indeed, with the same argument as for $L_1$, the derivative of $\varphi$ vanishes along $L_i$).
\item Otherwise, since $\varphi(x_3) \geq 0$, then on the one hand $x_3$ necessarily belongs to $H_+$ (possibly on its common boundary with $H_-$) and on the other hand $\varphi$ is necessarily negative on $L_2$. We distinguish between four subcases: (a) Assume that $x_4-x$ points in the direction opposite to $x_1-x$ and that there exists $x' \in L_3$ such that $\varphi(x') \leq 0$. Then $L_3$ is not colinear to $L_1$ and $\{ l_1,l_2 \} = \{ L_1,L_3 \}$ satisfies the required conditions. (b) Assume that $x_4-x$ points in the direction opposite to $x_1-x$ and that there is no $x' \in L_3$ such that $\varphi(x') \leq 0$. Then, the anticlockwise angle $\theta_3$ between by $L_2$ and $L_4$ is less than $\pi$. Let $x'$ be the intersection of $\tilde{L}_3$ with $L_3$. Then, $\varphi(x_4)\leq 0$, $\varphi(x_2)\leq 0$ and $\varphi(x')\geq 0$ since $x'\in L_3$ so $\{ l_1,l_2 \} = \{ L_1,\widetilde{L}_{3} \}$ satisfies the required conditions (in particular the two segments are not colinear). (c) Assume now that $x_4-x$ does not point in the opposite direction to $x_1-x$ and that either $x_4\in H_+$ or $x_4 \notin H_+$ and there is $x' \in L_4$ such that $\varphi(x') \geq 0$ then, as before, one can consider $\{l_1,l_2\}=\{L_1,L_4\}$. (d) Assume finally that $x_4-x$ does not point in the opposite direction to $x_1-x$, that $x_4\notin H_+$ and that there is no $x' \in L_4$ such that $\varphi(x') \geq 0$. Then, the anti-clockwise angle $\theta_1$ between $L_4$ and $L_2$ is less than $\pi$ and one can consider $\{ l_1,l_2 \} = \{ L_1,\widehat{L}_1 \}$. Indeed, remember that $\varphi$ is negative on $L_2$. Finally, $\widehat{L}_1$ goes through $x_1$ at which $\varphi$ is non-negative, and $\varphi$ is negative at both ends of $\widehat{L}_1$.
\end{itemize}
This completes the proof.
\end{proof}

\subsection{End of the proof of Proposition~\ref{p.eps_qi} via Kac-Rice estimates}\label{ss.Kac_Rice}

In this subsection we use results from Subsection~\ref{ss.pivo_excep} and Kac-Rice estimates to prove Proposition~\ref{p.eps_qi}. The only remaining step is the following proposition:
\begin{prop}\label{p.kac-rice}
Let $f$ be as in the statement of Proposition~\ref{p.eps_qi}. We use Notations~\ref{n.f_t} and~\ref{n.UandV}. There exist $C_1=C_1(\kappa)<+\infty$, $d_1=d_1(\kappa)<+\infty$ and $\eps_0=\eps_0(\kappa) \in ]0,1]$ such that, for all $p\in\R$ and $t\in [0,1]$, if $\eps \in ]0,\eps_0]$ and if $x \in \calV_1^\eps, \, y \in \calV^\eps_2$ are such that $|x-y| \geq d_1$ then:
\begin{itemize}
\item If neither $x \notin \calC_1$ nor $y \notin \calC_2$ then
\[
\prob  \left[X_t\in\piv_x(U^\eps)\cap \piv_y(V^\eps) \cond X_t(x)=X_t(y)=-p \right]\leq C_1(1+|p|)^4\eps^4\, .
\]
\item If among $x$ and $y$ one does not belong to $\calC_1 \cup \calC_2$ and the other belongs to $\calC_1 \cup \calC_2$ but is the corner of none of the $\calE_i$'s then:
\[
\prob \left[X_t\in\piv_x(U^\eps)\cap \piv_y(V^\eps) \cond X_t(x)=X_t(y)=-p \right]\leq C_1(1+|p|)^3\eps^3\, .
\]
\item If $x$ and $y$ both belong to $\calC_1 \cup \calC_2$ but are the corner of none of the $\calE_i$'s or if at least one of them does not belong to $\calC_1 \cup \calC_2$ then:
\[
\prob \left[X_t\in\piv_x(U^\eps)\cap \piv_y(V^\eps) \cond X_t(x)=X_t(y)=-p \right]\leq C_1(1+|p|)^2\eps^2\, .
\]

\item If $x$ or $y$ belongs to $\calC_1 \cup \calC_2$ but is the corner of none of the $\calE_i$'s then:
\[
\prob \left[X_t\in\piv_x(U^\eps)\cap \piv_y(V^\eps) \cond X_t(x)=X_t(y)=-p \right]\leq C_1(1+|p|)\eps\, .
\]
\end{itemize}
\end{prop}

Let us first wrap up  the proof of Propositon~\ref{p.eps_qi}.
\begin{proof}[Proof of Proposition~\ref{p.eps_qi}]
Remember that it is enough to prove~\eqref{e.sufficient}. First note that if $\eps\in ]\eps_0,1]$ (where $\eps_0$ is as in Propositon~\ref{p.kac-rice}) then the result is easily  obtained by bounding the probabilities by $1$. Now, assume that $\eps \in ]0,\eps_0]$. Then, by using Proposition~\ref{p.kac-rice}, we obtain that for the $O \left( \eps^{-4} \ar(\calK_1) \ar(\calK_2) \right)$ couples $(x,y)$ such that $x \in \calV_1^\eps \setminus \calC_1$ and $y \in \calV^\eps_2 \setminus \calC_1$, the quantitity $\prob \left[ X_t^\eps \in \piv_x(U^\eps)\cap\piv_y(V^\eps) \cond f_t(x)=f_t(y)=-p\right]$ is bounded by $C_1 (1+|p|)\eps^4$. Consequently, the sum over of all of these couples $(x,y)$ is bounded by $O \left( \eps^{-4} \ar(\calK_1) \ar(\calK_2) \right) (1+|p|)^4$. We reason similarly by also including the points on the boundary (which corresponds to $O\left( \eps^{-1} \leng(\calC_1) \right)$ points $x \in \calC_1$  and $O\left( \eps^{-1} \leng(\calC_2) \right)$ points $x \in \calC_2$) and at the corners (which correspond to $O(k_1)$ points $x \in \calV^\eps_1$ and $O(k_2)$ points $y \in \calV^\eps_2$).
\end{proof}
We now prove Proposition~\ref{p.kac-rice}.
\begin{proof}[Proof of Proposition~\ref{p.kac-rice}]
We prove the first item since the proof of the others is the same (possibly by using Lemma~\ref{l.boundary} instead if Lemma~\ref{l.interior}). Fix $t\in [0,1]$. Throughout the proof, the bounds will be uniform with respect to $t$. By combining Lemmas~\ref{l.qi.piv_to_geometry} and~\ref{l.interior}, we obtain that there exist a finite set of unit vectors $\mathfrak{V}$ independent of everything else, an absolute constant $C_0<+\infty$, and a finite set of $4$-uples of segments $\calL=\calL(x,y,\eps)$ such that $\card \calL \leq C_0$ and such that:
\bi 
\item For every $(l_1,l_2,l_1',l_2') \in \calL$ we have: The segments $l_1,l_2$ have non-colinear unit vectors $v_1,v_2 \in \mathfrak{V}$, are of length at most $C_0 \eps$, and are at distance at most $C_0$ from $x$. Moreover, the same holds for $l_1',l_2'$ near $y$ and with non-colinear unit vectors $v_1',v_2' \in \mathfrak{V}$.
\item The probability of the first item of Proposition~\ref{p.kac-rice} is no greater than the sum over all $(l_1,l_2,l_1',l_2') \in \calL$ of the expectation of:
\[ 
\textup{Card}\{ (a_1,a_2,b_1,b_2) \in l_1\times l_2\times l_1'\times l'_2 \, : \, \forall i,j \in \{ 1,2 \}, \, \partial_{v_i}f_t(a_i)=\partial_{v_j}f_t(b_j)=0\, \}\, .
\]
\ei
To control this expectation, we wish to apply the Kac-Rice formula. In order to do so we introduce the following notation. For each $(a_1,a_2,b_1,b_2)\in l_1\times l_2\times l_1'\times l'_2$, let
\begin{align*}
\Phi_t=\Phi_t(a_1,a_2,b_1,b_2)&=(\partial_{v_1}^2f_t(a_1),\partial_{v_2}^2f_t(a_2),\partial_{v_1'}^2f_t(b_1),\partial_{v_2'}^2f_t(b_2)) \, ,\\
\Psi_t=\Psi_t(x,y)&=(f_t(x),f_t(y)) \, ,\\
\Upsilon_t=\Upsilon_t(a_1,a_2,b_1,b_2)&=(\partial_{v_1}f_t(a_1),\partial_{v_2}f_t(a_2),\partial_{v_1'}f_t(b_1),\partial_{v_2'}f_t(b_2)) \, .
\end{align*}
Since $\kappa$ satisfies Condition~\ref{a.decay2}, then the covariance:
\[
D_t=D_t(a_1,a_2,b_1,b_2)=\left(\begin{matrix}
D_t^{11} & D_t^{12}\\ D_t^{21} & D_t^{22} \end{matrix}\right)
\]
of $(\Psi_t,\Upsilon_t)$ converges as $\eps\rightarrow 0$ and $|x-y|\rightarrow +\infty$, at a rate depending only on $\kappa$, to the following covariance:
\[
D_*=\left(\begin{matrix}
D_*^{11} & D_*^{12}\\ D_*^{21} & D_*^{22} \end{matrix}\right)=\left(\begin{matrix}
I_2 & 0\\ 0 & D_*^{22} \end{matrix}\right)
\]
 where:
\[
D_*^{22}=\left(\begin{matrix}
-\partial_{v_1}^2\kappa(0) & -\partial_{v_1}\partial_{v_2}\kappa (0) & 0 & 0\\
-\partial_{v_1}\partial_{v_2}\kappa (0) & -\partial_{v_2}^2\kappa(0) & 0 & 0\\
0 & 0 & -\partial_{v_1'}^2\kappa(0) & -\partial_{v_1'}\partial_{v_2'}\kappa (0)\\
0 & 0 & -\partial_{v_1'}\partial_{v_2'}\kappa (0) & -\partial_{v_2'}^2\kappa(0)
\end{matrix}\right)\, .
\]
Here we used Lemma~\ref{l.reg} and Remark~\ref{r.odd}. Since $v_1$ and $v_2$ (resp. $v_1'$ and $v_2'$) are non-colinear, the vectors $(\partial_{v_1}f(0),\partial_{v_2}f(0))$ and $(\partial_{v_1'}f(0),\partial_{v_2'}f(0))$ are non-degenerate (see Remark~\ref{r.bm}) so $D_*$ is non-degenerate. Consequently, there exist $C_2=C_2(v_1,v_2,v_1',v_2',\kappa)\in]0,+\infty[$, $d_1=d_1(v_1,v_2,v_1',v_2',\kappa)<+\infty$ and $\eps_0=\eps_0(v_1,v_2,v_1',v_2',\kappa) \in ]0,1]$ such that, if $\eps \in ]0,\eps_0]$ and $|x-y| \geq d_1$ then:
\begin{itemize}
\item the matrix $D_t^{11}$ is non-degenerate;
\item the matrix $\widetilde{D}_t=D_t^{22}-D_t^{21}(D_t^{11})^{-1}D_t^{12}$ is non-degenerate;
\item $\det(\widetilde{D}_t)\geq C_2^{-1}$;
\item the coefficients of $D_t^{-1}$ are no greater than $C_2$.
\end{itemize}
In addition, $\kappa$ is of class $C^8$ so Theorem~\ref{t.Kac_Rice} applies to the field $\Upsilon_t$ conditionned on $\Psi_t=(-p,-p)$. Since conditioning and differentiation 'commute' (see Remark~\ref{r.conditional_regression}), we obtain that the aforementioned expectation is no greater than:
\[
\int_{l_1\times l_2\times l_1'\times l'_2} \frac{\E\left[\prod_{i=1}^4|(\Phi_t)_i(a_1,a_2,b_1,b_2)| \cond \Psi_t(x,y)=(-p,-p)\, ,\, \Upsilon_t(a_1,a_2,b_1,b_2)=0\right]}{(2\pi)^2\sqrt{\det\left(\widetilde{D}_t(a_1,a_2,b_1,b_2)\right)}}da  \, db\, .
\]
The denominator is uniformly bounded from below by the previous discussion. We claim that if $\eps\leq\eps_0$ and $|x-y|\geq d_1$, the numerator is $O\left((1+|p|)^4\right)$. To prove this, notice first that $D_t$ is non-degenerate so Lemma~\ref{l.prod_Gaussian} applies. Moreover, the variance of the entries of $\Phi_t$ depends only on $\kappa$. All that remains is to bound its conditional mean. Firstly, the covariances of the entries of $\Phi_t$ and those of $(\Psi_t,\Upsilon_t)$ are bounded\footnote{Indeed, this follows from Lemma~\ref{l.reg} and the fact that for any two $L^2$ random variables $\xi_1$ and $\xi_2$, $\left|\E\left[\xi_1\xi_2\right]\right|\leq \frac{1}{2}\left(\E\left[\xi_1^2\right]+\E\left[\xi_2^2\right]\right)$.} by constants depending only on the derivatives up to order three of $\kappa$ at $0$. 
Moreover, $D_t^{-1}$ has bounded coefficients so the conditional mean of $\Phi_t$ is $O(|p|)$. Hence, by Lemma~\ref{l.prod_Gaussian}, the numerator is $O\left((1+|p|)^4\right)$. Finally, the integration domain has volume $O(\eps^4)$.
\end{proof}

\subsection{Completing the proof of Theorem~\ref{t.qi_rect}}\label{ss.components}

In this subsection we explain how to complete the proof of Theorem~\ref{t.qi_rect} to take into account events measureable with respect to the number of level lines components inside the rectangles $\calE_i$. In particular, this subsection is of no use for the proof of the RSW estimate Theorem~\ref{t.RSW}. The part of the proof of Theorem~\ref{t.qi_rect} detailed in Subsections~\ref{ss.pivo_excep} and \ref{ss.Kac_Rice} hinges on the two following ideas: first, that the crossing events can be approximated by discrete events and second, that the fact that a point $x$ is pivotal for a crossing events implies certain exceptional conditions on its neighbors whose probabilities are easy to control. To complete the proof of of Theorem~\ref{t.NS}, we justify that the discretization of the additional events is valid in Lemma~\ref{l.components.2} which in turn relies on Lemma~\ref{l.components.1}. Then, we prove that the additional pivotal events imply the cancellation of certain derivatives in Lemma~\ref{l.components.piv_to_geometry} and Lemma~\ref{l.components.white}. The rest of the proof relies on results from Section~\ref{s.weak-qi}.
\begin{remark}
Lemmas \ref{l.components.1} and \ref{l.components.2} below could be deduced from Proposition 6.1 of \cite{bm_17} and Theorem 1.5 of \cite{bm_17} respectively. However, since we do not need to control the rate of convergence when $\eps\rightarrow 0$, we do not need a quantitative discretization scheme so instead we present a simpler proof relying only on transversality arguments.
\end{remark}
\begin{lemma}\label{l.components.1}
Let $\calE\subseteq\R^2$ be a rectangle. Assume that the Gaussian field $f$ satisfies Condition~\ref{a.super-std} and that $\kappa$ is $C^6$. Fix $p\in\R$. Then, a.s. there exists a (random) constant $\eps_0>0$ such that for a.e. $\eps\leq \eps_0$, we have:
\item[i)] $\calT^\eps$ and $\calN_p$ intersect transversally,
\item[ii)] each edge of $\calT^\eps$ inside $\calE$ has at most two intersection points,
\item[iii)] any two distinct intersection points of a common edge $e$ are connected by a smooth path in $\calN_p$ inside the union of the two faces adjacent to $e$,
\item[iv)] for each connected component $\calC$ of $\calN_p$ there exists an edge $e$ of $\calT^\eps$ such that $\calC$ intersects $e$ exactly once and $e$ has no other intersection with the nodal set,
\item[v)] there is no edge of $\calT^\eps$ included in the boundary of $\calE^\eps$ that is intersected twice by $\calN_p$, where $\calE^\eps$ is (one of) the largest rectangle whose sides are integer multiples of $\eps$ such that $\calE^\eps \subseteq \calE$.
\end{lemma}

\begin{proof}[Proof of Lemma \ref{l.components.1}]
By Lemma~\ref{l.transversality.1}, $\calN_p$ is a.s. smooth and intersects $\partial\calE$ transversally. Let $w$ be a unit vector tangent to an edge of the lattice. We apply Lemma~\ref{l.transversality.2} to $T=\calE$, $g=(f,\partial_wf,\partial_w^2f)$ and $v=(0,0,0)$ ($g$ has bounded density by Remark~\ref{r.bm} and by stationarity). This shows that the  set of points $x\in\calN_p$ such that $T_x\calN_p$ is tangent to $w$ is a.s. discrete. We then simply apply Lemma \ref{l.transversality.3} to $\calC$ the union of connected components of $\calN_p$ intersecting $\calE$ (who are a.s. in finite number and a.s. do not intersect $0$, possibly modifying them outside of $\calE$ to make $\calC$ compact). This establishes assertions i), ii) and iii).

To show iv), first take $\eps$ smaller than the distance between any two distinct connected components of $\calN_p$ intersecting $\calE$ so that each edge $e$ can intersect at most one connected component. Assume that $\calC$ intersects each edge an even number of times. Then, it must stay in a union of a face and its three adjacent faces. If $\eps^2$ is much smaller than the area of the smallest connected component of $\calE\setminus\calN_p$ this cannot happen so iv) is satisfied.

In order to show v), use once again Lemma~\ref{l.transversality.1} in order to obtain that $\calN_p$ intersects the boundary of $\calE$ transversally and only finitely many times. This completes the proof.
\end{proof}

In the arguments below, we will need to discretize level lines of the field. To this end, let us introduce some notations.
\begin{notation}\label{n.components.UandV}
Let $\eps>0$, $p\in\R$ and $(\calE_i)_{1\leq i\leq k_1+k_2}$, $\calK_1$, $\calK_2$, $\calC_1$, $\calC_2$ and $N_p(i)$ be as in Theorem~\ref{t.qi_rect}. Let $\calV_1^\eps$ and $\calV_2^\eps$ as in Notation~\ref{n.f_t}. Color the plane as explained at the beginning of Section~\ref{s.weak-qi}. Given such a coloring, each face has either zero or two sides whose ends have opposite colors. If a face has two such sides, draw a segment joining the middle of these two sides. This produces a collection of polygonal lines on the plane. We denote by $\calN_p^\eps$ the union of these lines. For each $i\in\{1,\dots,k_1+k_2\}$, let $\calE_i^\eps$ be (one of) the largest rectangle whose sides are integer multiples of $\eps$ and such that $\calE_i^\eps\subseteq\calE_i$, let $N_p^\eps(i)$ be the number of connected components of $\calN_p^\eps$ contained in $\calE_i^\eps$. Let $A$ be an event in the $\sigma$-algebra defined by events of the form $\{N_p(i)=m\}$ where $i\in\{1,\dots,k_1\}$ and $m\in\N$. Let $A^\eps$ be the same event as $A$ but with the $N_p(i)$'s replaced by the $N_p^\eps(i)$'s. There exists $U^\eps\subseteq\R^{\calV_1^\eps\cup\calV_2^\eps}$ (resp. $V^\eps\subseteq\R^{\calV_1^\eps\cup\calV_2^\eps}$) such that $A^\eps=\{X^\eps\in U^\eps\}$. Note that by construction, the events $A$ and $B$ belong to the Boolean algebra generated by events of the form $\{N_p(i)\in S\}$ where $S\subseteq\N$.
\end{notation}

\begin{lemma}\label{l.components.2}
Assume that the Gaussian field $f$ satsifies Condition~\ref{a.super-std} and that $\kappa$ is $C^6$. We use Notation~\ref{n.components.UandV}. Then,
\[
\limsup_{\eps\rightarrow 0}\prob\left[\forall i\in\{1,\dots,k_1+k_2\},\, N_p(i)=N_p^\eps(i)\right]=1\, .
\]
\end{lemma}
\begin{proof}
We start with the following claim.
\begin{claim}\label{cl.components.2}
For each $i\in\{1,\dots,k_1+k_2\}$ a.s., for Lebesgue-a.e. small enough $\eps>0$, $N_p(i)=N_p^\eps(i)$.
\end{claim}
\begin{proof}
Fix $i\in\{1,\dots,k_1+k_2\}$. By points i) to iv) of Lemma~\ref{l.components.1}, a.s., for a.e. $\eps>0$ small enough, $\calN_p$ intersects $\partial\calE_i$ and $\calT^\eps$ transversally, each edge of $\calT^\eps$ included in $\calE_i^\eps$ is crossed at most twice and any two intersection points of the same edge are connected by $\calN_p$ inside one of its adjacent faces. Also, each connected component of $\calN_p$ must intersect an edge which is crossed exactly once by $\calN_p$.

In particular, the following is an equivalent definition of $\calN_p^\eps(i)$ for a.e. $\eps>0$ small enough: i) Let $F$ be a face of the lattice with two sides $e,e'$ that are intersected by $\calN_p$ exactly once and consider a path $\gamma$ included in $F \cap \calN_p$ that connects $e$ and $e'$. Then, replace $\gamma$ by a straight line as in Figure~\ref{f.tr_nodal} (case 1). ii) Let $F$ be a face of the lattice with two sides $e,e'$ that are intersected by $\calN_p$ exactly once, let $e''$ the third edge adjacent to $F$ and let $F'$ be the other face adjacent to $e''$. Also, consider a path $\gamma$ included in $(F \cup F') \cap \calN_p$ that connects $e$ and $e'$ and intersects $e''$ twice. Then, replace $\gamma$ by a straight line in Figure~\ref{f.tr_nodal} (case 2).

\begin{figure}[!h]
\begin{center}
\includegraphics[scale=0.55]{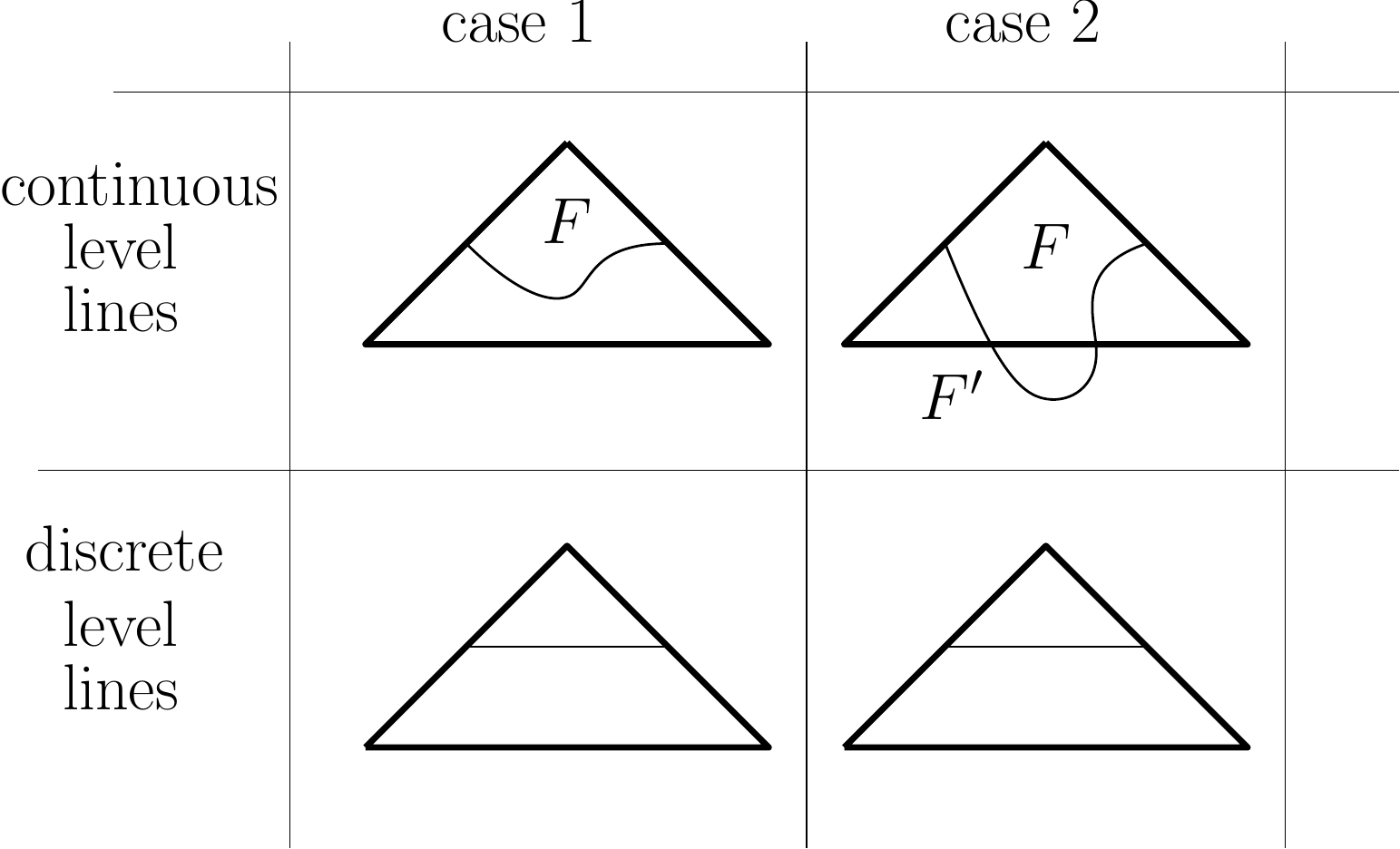}
\end{center}
\caption{An alternative definition of $\calN_p^\eps(i)$ when the conclusion of Lemma~\ref{l.components.1} holds.}\label{f.tr_nodal}
\end{figure}

One can see that, doing so, we redefine $\calN_p^\eps$ and this alternate definition shows that its connected components are naturally in bijection with those of $\calN_p$. Moreover, for all $eps>0$ small enough, connected components of $\calN_p$ included in $\calE_i$ are also included in $\calE_i^\eps$ so that $N_p(i)\leq N_p^\eps(i)$. On the other hand, if a continuous connected component gives rise to a discrete connected component included in $\calE_i^\eps$, it cannot cross edges of $\partial\calE_i^\eps$ once. But it cannot cross them twice either by point v) of Lemma \ref{l.components.1}. As a result, $N_p^\eps(i)\leq N_p(i)$.
\end{proof}


Let $\Xi(\eps)$ be the event that for all $i\in\{1,\dots,k_1+k_2\}$, $N_p(i)=N_p^\eps(i)$. Now, by Claim~\ref{cl.components.2}, for each $\delta>0$ there exists $\tau=\tau(\delta)>0$ such that, with probability at least $1-\delta$, for Lebesgue-a.e. $\eps\leq\tau$, $\Xi(\eps)$ is satisfied. Moreover, $\tau$ can be chosen so that $\lim_{\delta\rightarrow 0}\tau(\delta)=0$. In particular,
\[
\E\left[\int_0^\tau \un_{\Xi(\eps)}d\lambda(\eps)\right]\geq \tau(1-\delta)\, .
\]
By Fubini's theorem, we deduce that
\[
\int_0^\tau\prob\left[\Xi(\eps)\right]d\lambda(\eps)\geq\tau(1-\delta).
\]
In particular, there exists $\eps=\eps(\delta)\in]0,\tau(\delta)]$ such that $\prob\left[\Xi(\eps)\right]\geq 1-2\delta$. Since this holds for any $\delta>0$, the proof is complete.
\end{proof}

\begin{lemma}\label{l.components.piv_to_geometry}
Use Notation~\ref{n.components.UandV} and, for each $x\in\calV_1^\eps$, let $\omega^\eps\in\piv_x(U^\eps)$. Color the edges $e=(x,y)$ of $\calT^\eps$ such that $\omega^\eps(x),\omega^\eps(y)\geq -p$ in black and color the rest of the plane in white. Then:
\begin{enumerate}
\item if $x$ belongs to $\calK_1\setminus\calC_1$ then either the neighbors of $x$ are all of the same color or $x$ has (at least) four neighbors that have alternating color when listed in anti-clockwise order;
\item if $x$ belongs to $\calC_1$ but is not a corner, then it has three neighbors of alternating color when listed in anti-clockwise order.
\end{enumerate}
\end{lemma}
\begin{proof}
By Remark~\ref{r.piv}, we may assume that $A^\eps=\{N_p^\eps(i)=m\}$ for some $i\in\{1,\dots,k_1\}$ and $m\in\N$. Fix $\eps>0$, $x\in\calV^\eps_1$ and fix a value of $X^\eps$. If the set of neighbors has exactly one black connected component and one white component, then changing the color of $x$ does not change $N_p^\eps(i)$. Therefore $x$ being pivotal for $U^\eps$ implies the two items.
\end{proof}

The following lemma is a trivial application of Rolle's theorem.
\begin{lemma}\label{l.components.white}
Let $\varphi\in C^1(\R^2)$. Fix $x\in\R^2$ and assume that $\varphi(x)=0$. Then:
\begin{enumerate}
\item if there exist $x_1,x_2\in\R^2$ such that for each $i\in\{1,2\}$, $\varphi(x_i)\leq 0$ and such that $x\in ]x_1,x_2[$, then $\varphi|_{[x_1,x_2]}$ has a critical point;
\item if there exist $x_1,x_2,x_3,x_4\in\R^2$ such that for each $i\in\{1,2,3,4\}$, $\varphi(x_i)\leq 0$ and such that $l_1=[x_1,x_3]$ and $l_2=[x_2,x_4]$ intersect in their interior at $x$, then $\varphi|_{l_1}$ and $\varphi|_{l_2}$ have a critical point.
\end{enumerate}
\end{lemma}

We now complete the proof of Theorem~\ref{t.qi_rect}.
\begin{proof}[Proof of Theorem~\ref{t.qi_rect}: Part 2 of 2  Allowing components as well as crossings]
We use Notations~\ref{n.f_t} and \ref{n.components.UandV}. According to Lemma~\ref{l.components.2}
\[
\limsup_{\eps\rightarrow 0}\prob\left[\forall i\in\{1,\dots,k_1+k_2\},\, N_p(i)=N_p^\eps(i)\right]=1\, .
\]
We take a subsequence $(\eps_k)_{k\geq 1}$ along which the $\limsup$ is reached. Approximating crossings of the $\calE_i$ by discrete crossings of the $\calE_i^{\eps_k}$ we get $\lim_{k\rightarrow+\infty}\prob\left[A^{\eps_k}\bigtriangleup A\right]=0$. Therefore, it is enough to show that for $\eps$ small enough
\[
\left|\prob\left[A^\eps\cap B^\eps\right]-\prob\left[A^\eps\right]\prob\left[B^\eps\right]\right|\leq \frac{C}{\sqrt{1-\eta^2}}(1+|p|)^4e^{-p^2}\prod_{i=1}^2\left(\textup{Area}(\calK_i)+\textup{Length}(\calC_i)+1\right)
\]
for some constant $C=C(\kappa)<+\infty$. Here, unlike in Proposition~\ref{p.eps_qi}, $A$ and $B$ are events generated not only by crossing and circuit events but also by the $N_p(i)$'s. Nonetheless the proof is quite similar. Indeed, notice that Proposition~\ref{p.eps_qi} follows from Proposition~\ref{p.kac-rice} which in turn uses only the fact that for two points $x,y$ to be pivotal, certain derivatives of $f_t$ must vanish on certain deterministic segments. This is proved in Lemmas~\ref{l.qi.piv_to_geometry}, \ref{l.interior} and \ref{l.boundary}. In our case, first, we combine Lemma~\ref{l.qi.piv_to_geometry} with Lemma~\ref{l.components.piv_to_geometry}
using Remark~\ref{r.piv}. Then, we use Lemma~\ref{l.components.white} in addition to Lemmas~\ref{l.interior} and \ref{l.boundary}. The rest of the proof of Proposition~\ref{p.eps_qi} applies as is.
\end{proof}

\section{Tassion's RSW theory: the proof of Theorem~\ref{t.RSW}}\label{s.Tassion}

In this section, we prove Theorem~\ref{t.RSW} by relying on Sections~\ref{s.qi.vector} and~\ref{s.weak-qi} (but not on Subsection~\ref{ss.components}) and on~\cite{tassion2014crossing}. Our proof follows~\cite{tassion2014crossing} so instead of writing the details of each proof, we point out the steps of the original proof that need to be modified to work in our setting. We expect the reader to be familiar with~\cite{tassion2014crossing} and suggest that this section be read with said work at hand. Note that this simplifies the proof of~\cite{bg_16} since we can directly apply Tassion's method in the continuum instead of applying it to different discretizations of the model at each scale. We first prove the following weaker result:
\begin{prop}\label{p.first_rsw}
Let $f$ be a Gaussian field satisfying Conditions~\ref{a.super-std},~\ref{a.std},~\ref{a.decay2} as well as Condition~\ref{a.pol_decay} for some $\alpha > 4$. Let $\rho > 0$. There exists $c=c(\kappa,\rho)>0$ such that, for each $s > 0$, the probability that there is a left-right crossing of $[0,\rho s] \times [0,s]$ in $\calD_0$ is at least $c$.
\end{prop}
Throughout the proof, in~\cite{tassion2014crossing}, Tassion uses symmetries of the model such as stationarity (which is satisfied here by Condition~\ref{a.super-std}), symmetries, and the FKG inequality (which are also valid here by Condition~\ref{a.std} and Lemma~\ref{l.FKG}). The final ingredient of the proof is a quasi-independence lemma, which we will state when needed. Otherwise, the proof carries over with only minor changes due to the specificities of the model.
\begin{proof}
\textit{Step 1}: By Remark~\ref{r.planar_duality}, the probability that there is a left-right crossing of $[-s,s]^2$ is $1/2$ for any $s\in]0,+\infty[$. In particular, it is uniformly bounded from below  by some constant $c_0>0$, which is just Equation (1) of~\cite{tassion2014crossing}. In other words
\begin{equation}\label{e.tassion.1}
\forall s>0,\ \prob\left[\cross_0(s,s)\right]\geq c_0\, .
\end{equation}

\textit{Step 2}: Given $s\in]0,+\infty[$ and $\alpha,\beta\in[0,s/2]$ such that $\alpha<\beta$, we define the events $\calH_s(\alpha,\beta)$ and $\calX_s(\alpha)$ as follows (see Figure~\ref{f.fig_de_Tassion} below): The event $\calH_s(\alpha,\beta)$ is satisfied whenever there is a continuous path in $[-s/2,s/2]^2 \cap \calD_0$ connecting $\{-s/2\}\times[-s,s]$ to $\{s/2\}\times[\alpha,\beta]$. The event $\calX_s(\alpha)$ is the event that there is a path $\gamma_1$ in $[-s/2,s/2]^2 \cap \calD_0$ connecting $\{-s/2\}\times[-s/2,-\alpha]$ to $\{-s/2\}\times[\alpha,s/2]$, a path $\gamma_2$ in $[-s/2,s/2]^2 \cap \calD_0$ connecting $\{s/2\}\times[-s/2,-\alpha]$ to $\{s/2\}\times[\alpha,s/2]$ and a path in $[-s/2,s/2]^2 \cap \calD_0$ connecting $\gamma_1$ to $\gamma_2$. As in~\cite{tassion2014crossing}, we define $\phi_s:[0,s/2]\rightarrow[-1,1]$ as
\[
\phi_s(\alpha)=\prob[\calH_s(0,\alpha)]-\prob[\calH_s(\alpha,s/2)]\, .
\]
Then, Lemma 2.1 of~\cite{tassion2014crossing}, says that for each $s\in]0,+\infty[$, there is $\alpha_s\in[0,s/2]$ such that, for some $c_1>0$ independent of $s$,
\begin{equation}\label{e.tassion.2}
\forall \alpha\in[0,\alpha_s],\ \prob\left[\calX_s(\alpha)\right]\geq c_1; \forall \alpha\in[\alpha_s,s/2],\ \prob\left[\calH_s(0,\alpha)\right]\geq c_0/4+\prob\left[\calH_s(\alpha,s/2)\right]\, .
\end{equation}
To establish this inequality, Tassion uses the fact that $\phi_s$ is continuous and increasing and defines $\alpha_s$ using the preimage of $\phi_s$ of a certain value. Here, the continuity of $\phi_s$ follows easily from the fact the $f$ is a.s. continuous and that for each $x\in\R^2$, $\var(f(x))>0$. Moreover, the fact that $\phi_s$ is non-decreasing is immediate from its definition and this is sufficient for us since the argument works if one replaces $\min\left\{\phi_s^{-1}(s/4),s/4\right\}$ by $\sup\{\alpha\in]0,s/4[\ :\ \phi_s(\alpha)\leq c_0/4\}$. The rest of the proof of Lemma 2.1 uses only symmetries, the FKG inequality and Equation~\ref{e.tassion.1} and works as is.

\begin{figure}[!h]
\begin{center}
\hspace{2em}
\includegraphics[scale=0.6]{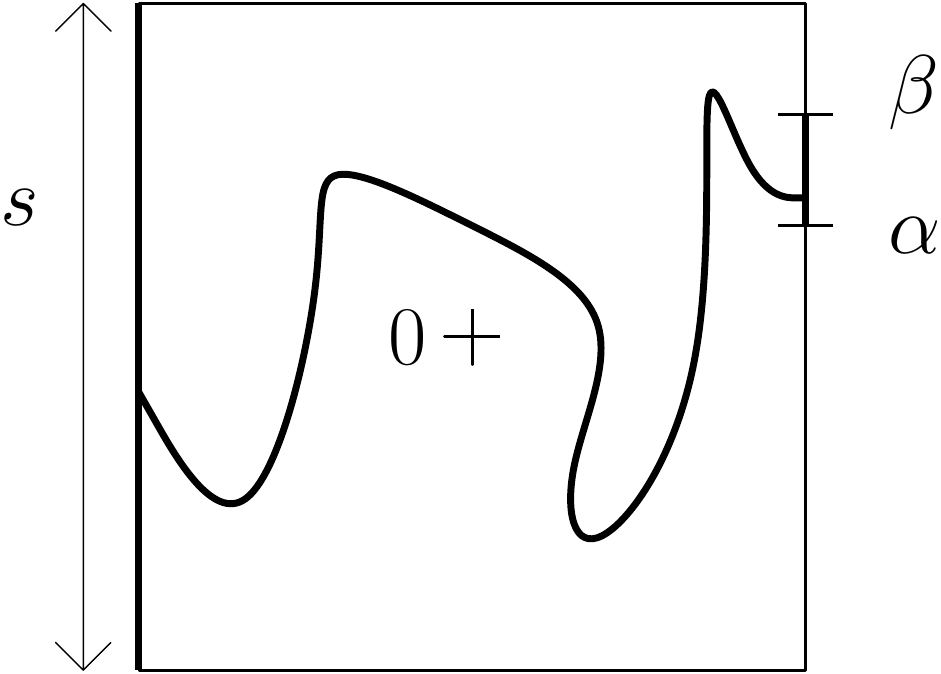} \hspace{2em} \includegraphics[scale=0.6]{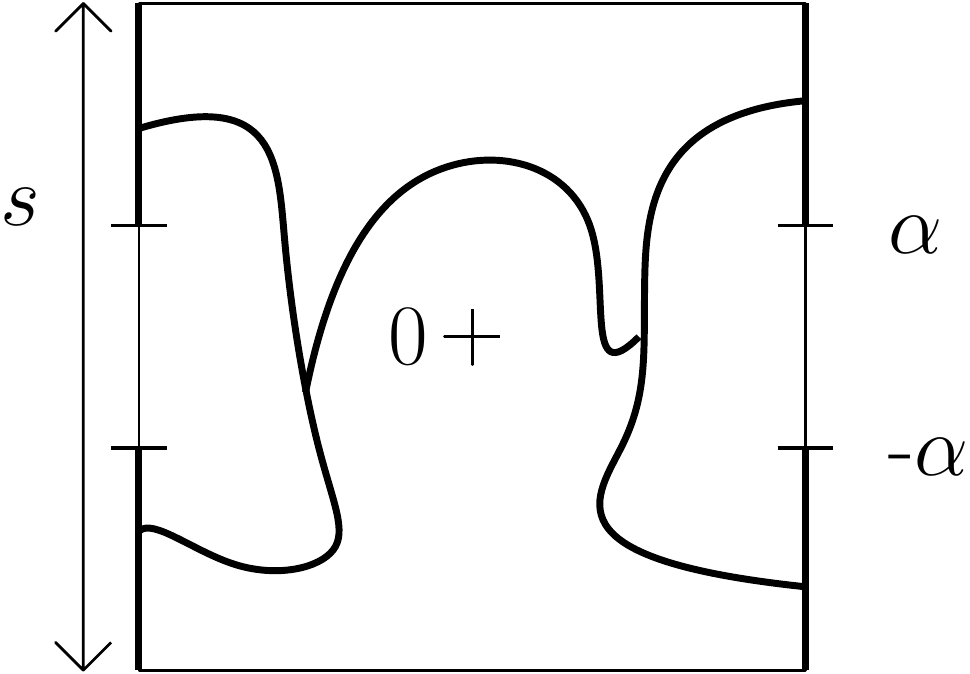}
\end{center}
\caption{The events $\mathcal{H}_s(\alpha,\beta)$ (left-hand-side) and 
$\mathcal{X}_s(\alpha)$ (right-hand-side). For every $\alpha \in [0,s/2]$, we let $\phi_s(\alpha)=\Pro [\calH_s (0,\alpha)] - \Pro [ \calH_s (\alpha, s/2)]$.}\label{f.fig_de_Tassion}
\end{figure}

\textit{Step 3}: For each $0<r<s$, let $\Circ_0(r,s)$ be the event that there is a circuit above level $0$ in the annulus $[-s,s]^2 \setminus ]-r,r[^2$ separating $[-s,s]^2$ from infinity. In Lemmas 2.2 and 3.1 of~\cite{tassion2014crossing}, Tassion shows that there exist constants $c_2,c_3\in]0,1[$ such that
\begin{equation}\label{e.tassion.3}
\forall s\geq 2,\ \alpha_s\leq 2\alpha_{2s/3}\Rightarrow \prob\left[\Circ_0(s,2s)\right]\geq c_2
\end{equation}
and
\begin{equation}\label{e.tassion.4}
\forall s\geq 1,\ t\geq 4s,\ \prob\left[\Circ_0(s,2s)\right]\geq c_2\text{ and }\alpha_t\leq s\Rightarrow \prob\left[\Circ_0(t,2t)\right]\geq c_3\, .
\end{equation}
The proof of these two lemmas relies only on the FKG, symmetries and Equations~\ref{e.tassion.2} so it carries over to our setting.\\

\textit{Step 4}: This is the step where Tassion uses a quasi-independence lemma. In our case, we will use the following direct consequence of Theorem~\ref{t.qi_rect}: 
\begin{cor}\label{c.quasi-ind-for-Tassion}
There exists a constant $C_0=C_0(\kappa)<+\infty$ such that, for every integer $N$ larger than $1$, for every $s \in [1,+\infty[$, for every $1 \leq r_1 \leq \cdots \leq r_N <+\infty$ such that $r_2 \geq r_1 + s$, and for every $B$ which belongs to the Boolean algebra generated by the events $\Circ_0(r_i,2r_i)$, $i = 2, \cdots, N$, we have:
\[
\big| \Pro \left[ \Circ_0(r_1,2r_1) \cap B \right] - \Pro \left[ \Circ_0(r_1,2r_1) \right] \Pro \left[ B \right] \big| \leq C_0 N  r_N^{4} \, s^{-\alpha} \, .
\]
\end{cor}
Using this corollary, we prove an analog of Lemma 3.2 of~\cite{tassion2014crossing}. Let us first introduce some notation. Given $c_0$ as in Step 1 and $c_2$ and $c_3$ as in Step 3, let $C_1<+\infty$ be such that $(1-c_3)^{\lfloor C_1/2\rfloor}<c_0/8$ and let $s_0<+\infty$ be such that for each $s\geq s_0$, $\frac{C_0}{c_3} \, \lfloor \log_5(C_1/2) \rfloor \, (C_1s)^4 s^{-\alpha} < c_0/8$ (where $C_0$ is as in Corollary~\ref{c.quasi-ind-for-Tassion}). Then, we prove the following:
\begin{lemma}\label{l.tassion}
Let $s\geq s_0$ such that $\prob\left[\Circ_0(s,2s)\right]\geq c_2$, then, there exists $s'\in [4s,C_1s]$ such that $\alpha_{s'}>s$.
\end{lemma}
\begin{proof}
In the proof of his Lemma~3.2, Tassion uses FKG and the symmetry properties, as well as what we call Equations~\ref{e.tassion.2} and \ref{e.tassion.3}. The only place where he uses a quasi-independence property is where he proves that, if $\Pro \left[ \Circ_0(5^i s,2 \cdot 5^i s) \right] \geq c_3$ for any $i \in \{ 0, \cdots, \lfloor \log_5(C_1/2) \rfloor \}$ and if $s \geq s_0$, then:
\[
\Pro \left[ \Circ_0(s,C_1 s) \right] > 1-c_0/4 \, .
\]
In what follows, we prove such a result and we refer to~\cite{tassion2014crossing} for the rest of the proof. First note that:
\[
\Circ_0(s,C_1 s) \subseteq \cup_{i=0}^{\lfloor \log_5(C_1/2) \rfloor} \Circ_0(5^is,2 \cdot 5^is) \, .
\]
Now, by Corollary~\ref{c.quasi-ind-for-Tassion} applied $\lfloor \log_5(C_1/2) \rfloor - 1$ times:
\begin{align*}
& \Pro \left[ \left( \cup_{i=0}^{\lfloor \log_5(C_1/2) \rfloor} \Circ_0(5^is,2 \cdot 5^is) \right)^c \right]\\
& = \Pro \left[ \cap_{i=0}^{\lfloor \log_5(C_1/2) \rfloor} \Circ_0(5^is,2 \cdot 5^is)^c \right]\\
& \leq (1-c_3) \times \Pro \left[ \cap_{i=1}^{\lfloor \log_5(C_1/2) \rfloor} \Circ_0(5^is,2 \cdot 5^is)^c \right] -  C_0 \, \lfloor \log_5(C_1/2) \rfloor \, (C_1s)^4 s^{-\alpha}\\
& \leq \cdots\\
& \leq (1-c_3)^{\lfloor \log_5(C_1/2) \rfloor} + C_0\left(\sum_{j\geq 0} (1-c_3)^j\right) \, \lfloor \log_5(C_1/2) \rfloor \, (C_1s)^4 s^{-\alpha}\\
& \leq (1-c_3)^{\lfloor \log_5(C_1/2) \rfloor} + \frac{C_0}{c_3} \, \lfloor \log_5(C_1/2) \rfloor \, (C_1s)^4 s^{-\alpha}\\
& < c_0/4 \, ,
\end{align*}
where the last inequality follows from the definition of $C_1$ and the fact that $s \geq s_0$.
\end{proof}

\textit{Step 5}: As explained in the proof of Lemma 3.3 of~\cite{tassion2014crossing} and the final comment that follows it, Proposition~\ref{p.first_rsw} now follows for $s$ large enough from Equations~\eqref{e.tassion.2} and \eqref{e.tassion.3} and Lemma~\ref{l.tassion} as well as standard gluing constructions that use only the FKG inequality and from symmetries. By the FKG inequality\footnote{More precisely, the events $\{f\geq 1\textup{ on }B\}$ can be approximated by increasing events depending on a finite sets of points, to which one can apply the FKG inequality.} Theorem~\ref{t.Pitt} applied to events of the form $\{ f \geq 1 \text{ on } B \text{ translated by some vector}\}$, we obtain that, for each $s>0$, $f$ takes only values larger than or equal to $1$ on $[-s,s]^2$ with positive probability.
\end{proof}

We now prove Theorem~\ref{t.RSW}.
\begin{proof}[Proof of Theorem~\ref{t.RSW}]
We prove the result for $\calN_0$. This is sufficient since $\calN_0 \subseteq \calD_0$ and since the result for $\calD_0$ for $s$ less than some fixed constant can easily be proved as in the end of Proposition~\ref{p.first_rsw}.\\
Let $\calQ$ be a quad and note that there exist $\delta=\delta(\calQ)>0$, $n=n(\calQ),m=m(\calQ) \in \Z_{>0}$ and two sequences $(\calE_i)_{i=1}^n$ and $(\calE'_j)_{j = 1}^{m}$ of $2\delta \times \delta$ and $\delta \times 2\delta$ rectangles such that: i) if each $\calE_i$ (resp. $\calE_j$) is crossed lengthwise then $\calQ$ is crossed and ii) $\underset{x \in \cup_{i=1}^n \calE_i, \, y \in \cup_{j=1}^m \calE_j'}{\inf} |x-y| \geq \delta$. For each $s>0$, write $A_s$ (resp. $B_s$) for the event that each $s\calE_i$ is crossed (resp. each $s\calE_j'$ is dual-crossed) lengthwise. By stationarity, $\frac{\pi}{2}$-rotation invariance and Remark~\ref{r.planar_duality}, the crossing events of each of the rectangles above and below $0$ are bounded from below by the constant $c=c(\kappa,2)>0$ from Proposition~\ref{p.first_rsw}. Consequently, by Lemma~\ref{l.FKG}, for each $s>0$, $\Pro [A_s]\geq c^n$ and $\Pro  [B_s]\geq c^m$. But now, by Theorem~\ref{t.qi_rect}, there exists $C=C(\kappa)<+\infty$ such that, for each $s > 0$:
\[
\Pro \left[ A_s \cap B_s \right] \geq \Pro \left[ A_s \right] \Pro \left[ B_s\right] - C \, (\delta s + 1)^{4-\alpha} \, n m \, .
\]
Since, $\alpha > 4$ we have $C \, (\delta s + 1)^{4-\alpha} \, n m \underset{s \rightarrow +\infty}{\longrightarrow} 0$ so the left-hand-side is bounded from below by a positive constant for $s$ sufficiently large. But $A_s\cap B_s$ clearly implies the crossing of $s\calQ$ by $\calN_0$.
\end{proof}
Now that we have established Theorem \ref{t.RSW}, we apply it to obtain two results which are well known in Bernoulli percolation. Namely, the polynomial decay of the one-arm event: Proposition~\ref{p.arm}, and the absence of unbounded clusters at criticality: Proposition~\ref{p.kesten}. We are going to use the following notation:
\begin{notation}
If $0<r<s<+\infty$, we write $\calA(r,s)=[-s,s]^2 \setminus ]-r,r[^2$ and we write $\arm_0(r,s)$ (resp. $\arm_0^*(r,s)$) the event that there is a continuous path in $\calD_0 \cap \calA(r,s)$ (resp. in $\calA(r,s) \setminus \calD_0$) from the inner boundary of $\calA(r,s)$ to its outer boundary. 
\end{notation}
We start with the following result:
\begin{prop}\label{p.arm}
Let $f$ be a Gaussian field that satisfying Conditions~\ref{a.super-std},~\ref{a.std},~\ref{a.decay2} as well as Condition~\ref{a.pol_decay} for some $\alpha > 4$. There exists $C=C(\kappa)<+\infty$ and $\eta= \eta(\kappa) > 0$ such that, for each $1 \leq r < s +\infty$:
\[
\Pro \left[ \arm_0(r,s) \right] = \Pro \left[ \arm_0^*(r,s) \right] \leq C \, (r/s)^\eta \, .
\]
\end{prop}
\begin{proof}
Remark~\ref{r.planar_duality} and the fact that $f$ is centered imply that $\Pro \left[ \arm_0(r,s) \right] = \Pro \left[ \arm_0^*(r,s) \right]$. So let us prove the result only for $\arm_0^*(r,s)$. First fix $h\in [1/2,1[$ to be determined later. For each $i \in \{ 0, \cdots, \lfloor \log_5(\frac{s^h}{2 \cdot r^{1-h}}) \rfloor \}$, let $\Circ_0(i)$ denote the event that there is a circuit at level $0$ in the annulus $\calA(5^i (rs)^{1-h},2 \cdot 5^i (rs)^{1-h})$. Note that:
\[
\arm_0^*(r,s) \subseteq \bigcap_{i=0}^{\lfloor \log_5(\frac{s^h}{2 \cdot r^{1-h}}) \rfloor} \Circ_0(i)^c \, .
\]
Next, note that by Theorem~\ref{t.RSW} and by the FKG inequality Lemma~\ref{l.FKG}, there exists $c=c(\kappa) \in ]0,1[$ such that for each $i \in \{ 0, \cdots, \lfloor \log_5(\frac{s^h}{2 \cdot r^{1-h}}) \rfloor \}$, $\Pro \left[ \Circ_0(i) \right] \geq c$. Next, use the quasi-independence result Theorem~\ref{t.qi_rect} $\lfloor \log_5(\frac{s^h}{2 \cdot r^{1-h}}) \rfloor$ times to obtain that, for some $C'=C'(\kappa)<+\infty$ we have:
\begin{align*}
& \Pro \left[ \bigcap_{i=0}^{\lfloor \log_5(\frac{s^h}{2 \cdot r^{1-h}}) \rfloor} \Circ_0(i)^c \right]\\
& \leq (1-c) \times \Pro \left[ \bigcap_{i=1}^{\lfloor \log_5(\frac{s^h}{2 \cdot r^{1-h}}) \rfloor} \Circ_0(i)^c \right] - C' \, \lfloor \log_5(\frac{s^h}{2 \cdot r^{1-h}}) \rfloor \, (1+s^4) \, (rs)^{-\alpha(1-h)}\\
& \leq \cdots\\
& \leq (1-c)^{\lfloor \log_5(\frac{s^h}{2 \cdot r^{1-h}}) \rfloor} + C'\left(\sum_{j\geq 0} (1-c)^j\right) \, \lfloor \log_5(\frac{s^h}{2 \cdot r^{1-h}}) \rfloor \, (1+s^4) \, (rs)^{-\alpha(1-h)}\\
& \leq (1-c)^{\lfloor \log_5(\frac{s^h}{2 \cdot r^{1-h}}) \rfloor} + \frac{C'}{c} \, \lfloor \log_5(\frac{s^h}{2 \cdot r^{1-h}}) \rfloor \, (1+s^4) \, (rs)^{-\alpha(1-h)} \, .
\end{align*}
Since $\alpha>4$, we can find $h$ sufficiently small to obtain what we want. 
\end{proof}

From Proposition~\ref{p.arm} we get the following analog of the celebrated theorem by Harris~\cite{harris1960lower} (which states that, for Bernoulli percolation on $\Z^2$ with parameter $1/2$, there is no infinite cluster). This result was also obtained by Alexander in~\cite{alex_96} for stationary ergodic planar
fields satisfying an FKG-type inequality under some mild non-degeneracy assumptions.

\begin{prop}\label{p.kesten}
With the same hypotheses as Proposition~\ref{p.arm}, a.s. every connected
 component of $\calD_0$ is bounded.
\end{prop}
\begin{proof}
By a union-bound and translation invariance, it is enough to prove that a.s. there is no unbounded component of $\calD_0$ which intersects $[-1,1]^2$, which is the case since by Proposition~\ref{p.arm}, $\Pro \left[ \arm_0(1,s) \right]$ goes to $0$ as $s$ goes to $+\infty$.
\end{proof}
The natural question arising from this proposition is whether or not this remains true for $\calD_p$ with $p>0$. This is the object of~\cite{rv_bf} where we prove that, for the Bargmann-Fock field, there is a unique unbounded connected component in $\calD_p$ as soon as $p>0$, thus obtaining the analogue of Kesten's famous theorem~\cite{kesten1980critical} (which states that the critical point for Bernoulli percolation on $\Z^2$ is $1/2$).

\section{Concentration from below of the number of nodal lines: the proof of Theorem~\ref{t.concentration_NS}}\label{s.concentration_NS}

In this section, we prove Theorem~\ref{t.concentration_NS} by using Theorem~\ref{t.NS} and our quasi-independence result Theorem~\ref{t.qi_rect}. The idea of the proof is the following. Let $\eps > 0$. We first tile the square $[-s/2,s/2]^2$ with $(r/s)^2$ mesoscopic squares of size $r$. Then, we use Theorem~\ref{t.NS} and our quasi-independence result Theorem~\ref{t.qi_rect} to prove that the density of $r \times r$ squares containing less than $r^2(c_{NS}-\eps)$ nodal components is asymptotically small. More precisely, we will note that, if the number of such squares is greater than $\delta(s/r)^2$, then there exist $\delta (s/r)^2/8$ such squares that are at distance at least $r$ from each other. By Theorem~\ref{t.qi_rect}, this has probability $\Pro \left[ \frac{N_0(r)}{r^2}-c_{NS} \leq -\eps \right]^{\delta(s/r)^2/8}$ up to errors involving terms of the form  $\sup_{x \, : \, |x| \geq r} |\kappa(x)|$. The last step is an optimization on the choice of $r$.\\
Upper concentration on the other hand seems to require some control of the tail of the density of nodal components. For the moment, it is not even known whether this density is $L^2$. This type of information seems necessary for the following reason. In Item (1) of Theorem~\ref{t.concentration_NS}, for instance, we consider exponential concentration of the density of components. To this end we write the number of components as a sum of quasi-independent random variables. But a direct consequence of Cramér's theorem is that, if $X_1, X_2, \cdots$ are i.i.d. $L^1$ positive random variables such that $\E \left[ e^{\theta X_1} \right] =+\infty$ for every $\theta > 0$, then $\left( \frac{X_1+\cdots+X_n}{n} \right)_n$ does not have exponential concentration around its mean. Note finally that to have an upper bound concentration, we need to take care of the mesoscopic components that intersect several $r \times r$ squares. However, these do not add any difficulty. Indeed, by~\cite{nazarov2015asymptotic}, if we write $N'(r)$ for the number of nodal components which \textit{intersect} a $r \times r$ box (and are not just included) then Theorem~\ref{t.NS} also holds for $N'(r)$ (with the same constant $c_{NS}$).

\begin{figure}[!h]
\begin{center}
\hspace{2em}
\includegraphics[scale=0.8]{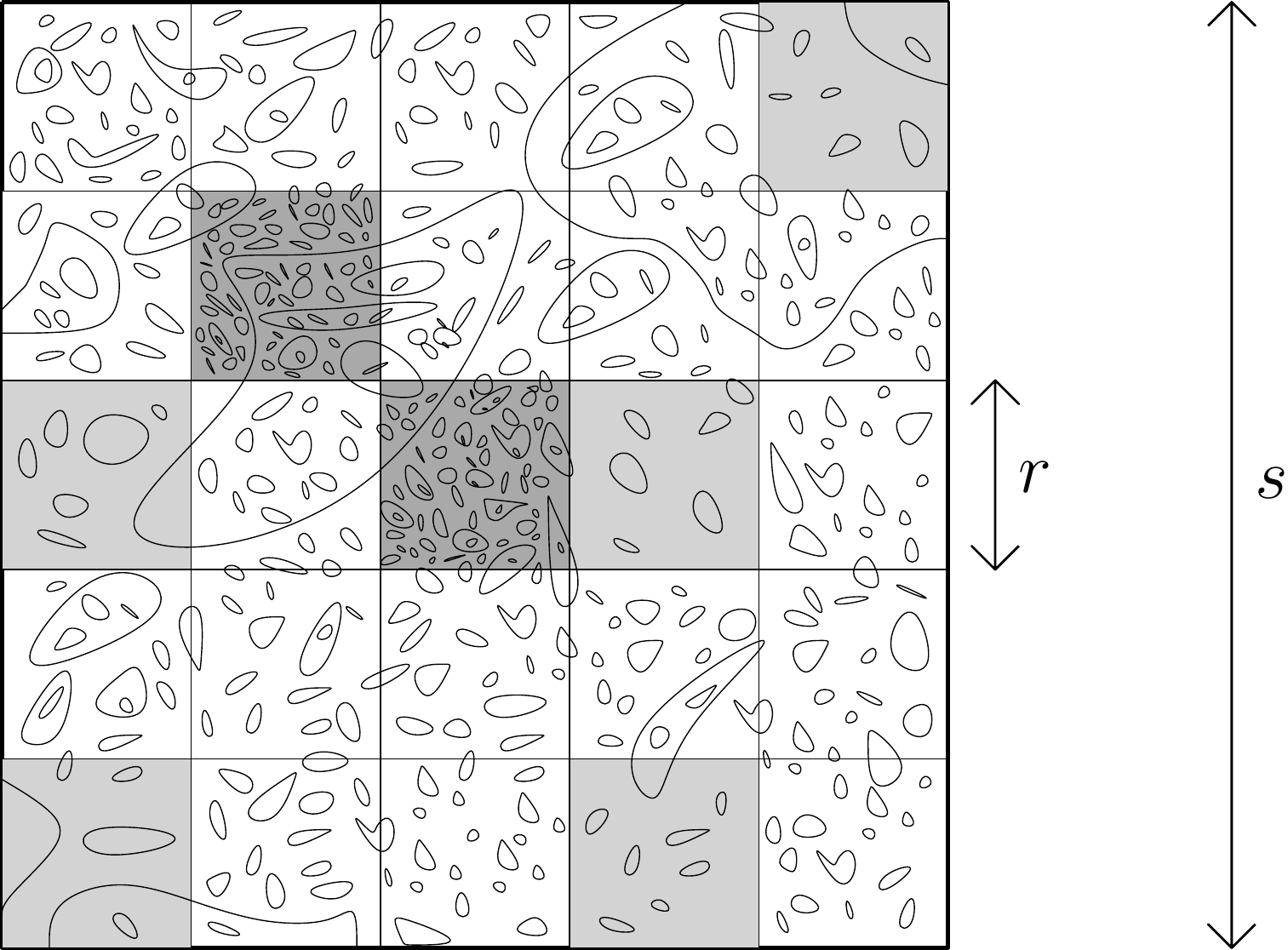}
\end{center}
\caption{The components of $\calN_0$ in $[-s/2,s/2]^2$. In light gray: the $r \times r$ squares in which the density of components is smaller than expected. Combining Theorem~\ref{t.NS} with Theorem~\ref{t.qi_rect}, we prove that that with high probability there are not too many such squares. In dark gray, the $r \times r$ squares in which the number of components is much greater than expected. Since we do not know whether or not the density of nodal component has an heavy tail, it is very hard to control these exceptional squares.}\label{f.components}
\end{figure}

\begin{proof}[Proof of Theorem~\ref{t.concentration_NS}]
Assume that $f$ is a planar Gaussian field satisfying Conditions~\ref{a.super-std},~\ref{a.decay2} and~\ref{a.spectral}. First note that it is sufficient to prove the result for $\eps$ sufficiently small and fix $\eps \in ]0,c_{NS}/2[$. Let $1 \leq r \leq s$ be such that $s \in r\N$ and tile the square $[-s/2,s/2]^2$ with $(s/r)^2$ $r \times r$ squares $S_1, \cdots, S_{(s/r)^2}$. Throughout the proof, we take the liberty of omitting floor functions. For each $t \in [0,+\infty[$, write $\kappa_t=\sup \{ |\kappa(x)| \,  : \, |x| \geq t \}$.\\
By Theorem~\ref{t.NS}, for each $h\in]0,1/2[$, there exist $r_0=r_0(\eps,h) < +\infty$ such that, if $r \geq r_0$, then:
\begin{equation}\label{e.NSeps2}
\Pro \left[ \frac{N_0(r)}{r^2} \leq c_{NS}-\eps \right] \leq h \, ,
\end{equation}
We also assume that $r_0$ is sufficiently big so that $ \kappa_{r_0} \leq 1/2$ and we assume that $r \geq r_0$. For every $i \in \{ 1, \cdots, (s/r)^2 \}$, write $N_0^i$ for the number of connected components of $\calN_0$ included in $S_i$ and note that, if $\frac{N_0(s)}{s^2} \leq c_{NS}-2\eps$, then there exist $(s/r)^2\frac{\eps}{c_{NS}-\eps}$ squares $S_i$ such that $\frac{N_0^i}{r^2}\leq c_{NS}-\eps$. As a result, if $\eta=\eta(\eps)=\frac{1}{8}\times\frac{\eps}{c_{NS}-\eps}$, there exist $\eta \cdot (s/r)^2$ squares $S_i$ at distance at least $r$ from each other and such that $\frac{N_0^i}{r^2} \leq c_{NS} -\eps$. Let $S_{i_1}, \cdots, S_{i_n}$ be $\eta \cdot (s/r)^2$ pairwise distinct squares among the $(s/r)^2$ squares at distance at least $r$ from each other. In the following, we estimate the probability that for each $j \in \{ 1, \cdots, \eta \cdot (s/r)^2 \}$, $\frac{N_0^{i_j}}{r^2}\leq c_{NS}-\eps$. Recall that $h<1/2$ and that $0<\eps<c_{NS}/2$ so $0<\eta<1/8$. By Theorem~\ref{t.qi_rect} applied $\eta \cdot (s/r)^2-1$ times, by translation invariance and by~\eqref{e.NSeps2}:
\begin{align*}
&\Pro \left[ \forall j \in \{ 1, \cdots, \eta \cdot (s/r)^2 \}, \, \frac{N_0^{i_j}}{r^2} \leq c_{NS}-\eps \right]\\
& \leq h \times \Pro \left[  \forall j \in \{ 2, \cdots, \eta \cdot (s/r)^2 \}, \, \frac{N_0^{i_j}}{r^2} \leq c_{NS}-\eps \right] + O \left( \kappa_r \, r^2 (s^2 + \eta \cdot (s/r)^2) \right)\\
& \leq h \times \Pro \left[  \forall j \in \{ 2, \cdots, \eta \cdot (s/r)^2 \}, \, \frac{N_0^{i_j}}{r^2} \leq c_{NS}-\eps \right] + O \left( \kappa_rr^2s^2\right)\\
& \leq \cdots\\
& \leq h^{\eta \cdot (s/r)^2}+ \,\left(\sum_{j\geq 0} h^j\right)O\left( \kappa_r \, r^2s^2 \right)= h^{\eta \cdot (s/r)^2}+ \, O\left( \kappa_r \, r^2s^2 \right)\\
& \leq h^{\eta \cdot (s/r)^2}+ \,\frac{1}{1-h}O\left( \kappa_r \, r^2s^2 \right)= h^{\eta \cdot (s/r)^2}+ \, O\left( \kappa_r \, r^2s^2 \right) \, .
\end{align*}
where the constants in the $O$'s depend only on $\kappa$. As a result :
\begin{eqnarray}
\Pro \left[ \frac{N_0(s)}{s^2} \leq c_{NS}-2\eps \right] & \leq & \binom{(s/r)^2}{\eta \cdot (s/r)^2} \left( h^{\eta \cdot (s/r)^2}+ O \left(\kappa_r \, r^2s^2 \right) \right) \nonumber\\
& \leq &  (2h^\eta)^{(s/r)^2} + O \left( 2^{(s/r)^2}\,\kappa_r \, r^2s^2 \right) \, . \label{e.nearly_finished}
\end{eqnarray}
Let us first treat the case of Item~1 i.e. assume that there exists $C<+\infty$ and $c>0$ such that $\kappa_r \leq C\exp(-cr^2)$. Then, the right hand side of \eqref{e.nearly_finished} is
\[
(2h^\eta)^{(s/r)^2}+\, O\left(\exp\left((s/r)^2\log(2)-cr^2+4\log(s)\right)\right) \, .
\]
Taking $h=h(\eta)$ small enough and $r=M\sqrt{s}$ for $M=M(c)$ large enough, this quantity is exponentially small in $s$ so we are done.\\

Let us now treat the case of Item~2 i.e. assume that there exists $C<+\infty$ and $\alpha>4$ such that $\kappa_r \leq C r^{-\alpha}$. Then, the right hand side of \eqref{e.nearly_finished} is
\[
(2h^\eta)^{(s/r)^2}+\, O\left(2^{(s/r)^2}s^2r^{2-\alpha}\right) \, .
\]
Fix $\delta > 0$. Choosing $r=s/\sqrt{a\log_2(s)}$ for $a=a(\delta)>0$ small enough, the second term in the sum is $O\left(s^{4-\alpha+\delta}\right)$. Having chosen $a$, we choose $h=h(a,\eta)$ such that the first term is also $O\left(s^{4-\alpha+\delta}\right)$. Since this is true for any $\eps\in]0,c_{NS}/2[$ and any $\delta>0$, we are done.
\end{proof}

\begin{remark}
Note that we have used Theorem~\ref{t.NS} only to obtain~\eqref{e.NSeps2}. Hence, our lower concentration result Theorem~\ref{t.concentration_NS} holds if, instead of Condition~\ref{a.spectral} (which is the assumption to apply Theorem~\ref{t.NS}), we assume that there exists a constant $c_{NS}=c_{NS}(\kappa) \in\, ]0,+\infty[$ such that, for each $\eps > 0$, $\Pro \left[ \frac{N_0(s)}{s^2} \leq c_{NS}-\eps \right]$ goes to $0$ as $s$ goes to $+\infty$.
\end{remark}

\appendix

\section{Classical tools}\label{s.basics}

In this section we present classical or elementary results about Gaussian vectors and fields.

\subsection{Classical results for Gaussian vectors and fields}

\paragraph{Differentiating Gaussian fields.} When one consider derivatives of Gaussian fields, it is important to have the following in mind (see for instance Appendices A.3 and A.9 of~\cite{nazarov2015asymptotic}):
\begin{lemma}\label{l.reg}
Let $f$ be an a.s. continuous Gaussian field with covariance\footnote{Here and below, $C^{l,l}$ means that all partial derivatives of $K$ which include at most $l$ differentiations in the first variable and $l$ differentiations in the second variable exist and are continuous.} $K\in C^{k+1,k+1}(\R^n\times\R^n)$ and mean $\mu\in C^k(\R^n)$. Then, $f$ is almost surely $C^k$. Conversely, if a.s. $f$ is $C^k$, then $K \in C^{k,k}$, $\mu\in C^k$ and for every multi-indices $\beta,\gamma \in \N^2$ such that $\beta_1+\cdots+\beta_n \leq k$ and $\gamma_1+\cdots+\gamma_n \leq k$, we have:
\[
\cov \left( \partial^\beta f(x), \partial^\gamma f(y) \right) = \E \left[ (\partial^\beta f(x)-\partial^\beta\mu(x)) (\partial^\gamma f(y)-\partial^\gamma\mu(y)) \right] =  \partial^\beta_x\partial^\gamma_y K(x,y) \, .
\]
\end{lemma}
\begin{remark}\label{r.odd}
Lemma~\ref{l.reg} has the following consequence: if $f$ satisfies Condition~\ref{a.super-std} and is a.s. $C^1$ then, for each $\beta\in\N^2$ such that $\beta_1+\beta_2$ is odd, $\partial^\beta\kappa(0)=0$.
\end{remark}
\begin{remark}\label{r.bm}
Another consequence of Lemma~\ref{l.reg} is that if $f$ is a.s. $C^1$ and satisfies Condition~\ref{a.super-std} then for each $x \in \R^2$ and for $v,w$ non-colinear unit vectors, the Gaussian vector $(\partial_vf(x),\partial_wf(x))$ is non-degenerate. Indeed, if this was not the case, then we would obtain the existence of some non-zero vector $u$ such that $\partial_uf$ would a.s. vanish identically, which would contradict the fact that $f$ is non-degenerate. Similarly, if $f$ is a.s. $C^2$ and satisfies Condition~\ref{a.super-std} then for each $x \in \R^2$ and each non-zero vector $w \in \R^2$, $(f(x),\partial_wf(x),\partial_w^2f(x))$ is non-degenerate. Indeed, $\partial_wf(x)$ is independent of the two other coordinate by Remark~\ref{r.odd} and if $(f(x),\partial_w^2f(x))$ were degenerate then as above this would contradict the fact that $f$ is non-degenerate.
\end{remark}

\paragraph{A FKG inequality for Gaussian vectors.} The following result by~\cite{pitt1982positively} says that positively correlated Gaussian vectors satisfy positive association. This is a key result when one wants to use Russo-Seymour-Welsh type techniques. We first need to introduce the following terminology: if $I$ is some set and $A \subseteq \R^I$ then we say that $A$ is increasing if for every $\omega \in A$ and every $\omega' \in \R^I$ such that $\omega'(i) \geq \omega(i)$ for every $i \in I$, we have $\omega' \in A$.

\begin{thm}[\cite{pitt1982positively}]\label{t.Pitt}
Let $(X_k)_{1\leq k\leq n}$ be a Gaussian vector such that, for every $k,l\in\{1,\dots,n\}$, $\E[X_kX_l]\geq 0$. Then, For every $A,B\subseteq\R^n$ increasing Borel subsets:
\[
\prob[X\in A\cap B]\geq \prob[X\in A] \, \prob[X\in B]\, .
\]
\end{thm}
This type of inequality is known as the \textbf{Fortuyn-Kasteleyn-Ginibre} (or \textbf{FKG}) inequality. Pitt's result easily generalizes to crossing and circuits events by approximation, one just needs to take care that the approximating events are increasing, see Lemma~\ref{l.FKG}.
\paragraph{Some basic lemmas.} The following lemma is useful to bound the expectation of the product of Gaussian variables. The first lemma is known as the regression formula and is quite classical in the field.
\begin{lemma}[Proposition 1.2 of \cite{azws}]\label{l.regression}
Let $(X,Y)$ be an $n+m$-dimensional centered Gaussian vector with covariance
\[
\left(\begin{matrix} A & B\\ B^t & D\end{matrix}\right)
\]
where $A$ (resp. $D$) is the covariance of $X$ (resp. $Y$). Assume $Y$ is non-degenerate. Then, the law of $X$ conditioned on $Y$ is that of a Gaussian vector with covariance $A-BD^{-1}B^t$ and mean $BD^{-1}Y$.
\end{lemma}
The next lemma is a simple application of the regression formula to the computation of conditional moments of Gaussian vectors.
\begin{lemma}\label{l.prod_Gaussian}
Let $(X,Y)$ be a centered Gaussian vector in $\R^n\times\R^m$ with covariance
\[
\left(\begin{matrix} A & B\\
B^t & D\end{matrix}\right)\, .
\]
Assume that $D$ is non-degenerate. Let $\mu\in\R^m$. Then, there exists $C=C(n)<+\infty$ such that
\[
\E\left[\prod_{i=1}^n |X_i|\cond Y=\mu\right]\leq C\max_{i\in\{1,\dots,n\};\, j,k\in\{1,\dots,m\}}\left(\sqrt{A_{ii}}\vee |B_{ik}D^{-1}_{kj}\mu_j|\right)^n
\]
\end{lemma}
\begin{proof}
By the regression formula (Lemma~\ref{l.regression}), $X$ conditioned on $Y=\mu$ has the law of a Gaussian vector $Z$ with covariance $\widetilde{A}=A-BD^{-1}B^t$ and mean $\widetilde{\mu}=BD^{-1}\mu$. Note that $BD^{-1}B^t$ is symmetric semi-definite. Therefore, its diagonal coefficients must be non-negative. Therefore, for each $i\in\{1,\dots,n\}$, $\widetilde{A}_{ii}\leq A_{ii}$. Moreover, for each $i\in\{1,\dots,n\}$, $|\widetilde{\mu}_i|\leq n^2 \max_{j,k\in\{1,\dots,m\}}|B_{ik}D^{-1}_{kj}\mu_j|$. The lemma then follows from the elementary observation that for each $n\geq 1$ there exists $C=C(n)<+\infty$ such that for each Gaussian vector $Z$ with covariance $\widetilde{A}$ and mean $\widetilde{\mu}$,
\[
\E\left[\prod_{i=1}^n |Z_i| \right]\leq C\max_{i\in\{1,\dots,n\}}\left(\sqrt{\widetilde{A}_{ii}}\vee |\widetilde{\mu}_i|\right)^n\, .
\]
\end{proof}

\begin{remark}\label{r.conditional_regression}
From Lemmas~\ref{l.regression} and \ref{l.reg} we deduce that if $f$ is an a.s. continuous and non-degenerate Gaussian field on $\R^n$ with $C^{k+1,k+1}$ covariance and $C^k$ mean and if $x_1,\dots,x_k\in\R^n$ are such that $(f(x_1),\dots,f(x_k))$ is a non-degenerate Gaussian vector, then, for each $v\in\R^n$ conditionally on $(f(x_1),\dots,f(x_k))=v$, $f$ is a Gaussian field with $C^{k+1,k+1}$ covariance and $C^k$ mean. Moreover, the covariance (resp. mean) of the derivatives of the conditional field is equal to the covariance (resp. mean) of the derivatives of the field under the same conditioning.
\end{remark}

\paragraph{A Kac-Rice formula.} The following result is a Kac-Rice type formula, which is for instance a particular case of Theorem~6.2 of~\cite{azws} (together with Proposition~6.5 therein):
\begin{thm}\label{t.Kac_Rice}
Let $\eps \in ]0,+\infty[$, let $n \in \Z_{>0}$, and let $\Phi_1,\cdots,\Phi_n$ denote $n$ continuous Gaussian fields $\, : \, [0,\eps] \rightarrow \R$ that are a.s. $C^2$ on $]0,\eps[$ and such that, for every $s \in [0,\eps]^n$, $\Phi(s)=(\Phi_1(s_1), \cdots,\Phi_n(s_n))$ is non-degenerate. Then
\[
\E \left[ \text{\textup{Card}} \{ s \in [0,\eps]^n \, : \, \Phi(s)=0 \} \right]
\]
equals:
\[
\int_{]0,\eps[^n} \varphi (s) \, \E \left[ \prod_{i=1}^n \left| \Phi_i'(s_i) \right| \cond \Phi(s)=0  \right] ds \, ,
\]
where $\varphi(s)$ is the density of $\Phi(s)$ evaluated at $0$.
\end{thm}

\subsection{Transversality of the level set and a non-quantitative discretization lemma}

In this subsection, we state transversality results which are quite classical in the field and which are very helpful to obtain some continuity results about crossing events. We also prove a non-quantitative discretization lemma useful to justify discrete approximation of certain events.

\begin{lemma}\label{l.transversality.1}
Assume that $f$ satisfies Condition~\ref{a.super-std} and that $\kappa$ is $C^6$. Fix $p\in\R$ and fix $(\gamma(t))_{t \in [0,1]}$ a smooth path in the plane. Then:
\bi 
\item[1.] A.s. $f^{-1}([-p,+\infty[)=:\calD_p$ and $f^{-1}(]-\infty,-p])$ are two $2$-dimensional smooth sub-manifolds of $\R^2$ with boundary. Moreover, a.s. their boundaries are equal and are the whole set $\calN_p$.
\item[2.] A.s., $\calN_p$ intersects $\gamma$ transversally.
\ei
\end{lemma}

To prove Lemma~\ref{l.transversality.1}, we can use the following lemma:

\begin{lemma}[see Lemma 11.2.10 of \cite{adta_rfg}]\label{l.transversality.2}
Let $n\in\N$. Let $T$ be a compact subset of $\R^n$ with Hausdorff dimension $k \in \N$. Let $g=(g_j)_{1\leq j\leq k+1}:\R^n\rightarrow\R^{k+1}$ be a Gaussian field that is a.s. $C^1$. Assume also that $g$ has a bounded density on $T$. Then, for each $v\in \R^{k+1}$, $g^{-1}(v)\cap T$ is a.s. empty.
\end{lemma}
\begin{proof}[Proof of Lemma \ref{l.transversality.1}]
First note that the fact that $\kappa$ is $C^6$ implies that $f$ is $C^2$ by Lemma~\ref{l.reg}. To prove the first part of the lemma, we fix $R \in ]0,+\infty[$ and $p\in\R$ and apply Lemma \ref{l.transversality.2} to $T=[-R,R]^2$ (of Hausdorff dimension $2$) and $g=(f,\partial_1 f,\partial_2 f)$ with $v=(-p,0,0)$. For every $x$, we have the following: i) by Remark \ref{r.odd}, $f(x)$ is independent of $(\partial_1f(x),\partial_2f(x))$ and ii) by Remark~\ref{r.bm}, $(\partial_1f(x),\partial_2f(x))$ is non-degenerate. As a result, $g(x)$ is non-degenerate. Since $g$ is stationary, this implies that $g$ has bounded density. We obtain that a.s. $f$ vanishes transversally on $\calN_p\cap[-R,R]^2$. By taking the intersection of such events for $R=1,2,\cdots$ we end the proof of the first statement. For the second part of the statement, we apply Lemma~\ref{l.transversality.2}, this time for $T=\{ \gamma(t) \}_{t \in [0,1]}$ (of Hausdorff dimension $1$) and $g(t)=((f\circ\gamma)(t),(f\circ\gamma)'(t))$ with $v=(-p,0)$. As before, for every $t$, $g(t)$ is non-degenerate. By continuity of the covariances, this implies that $g$ restricted to $T$ has bounded density, so Lemma~\ref{l.transversality.2} does apply.
\end{proof}

\begin{remark}\label{r.planar_duality}
The following can easily be deduced from Lemma~\ref{l.transversality.1}: Assume that $f$ satisfies Condition~\ref{a.super-std} and that $\kappa$ is $C^3$. Fix $p \in \R$ and let $\calQ\subseteq\R^2$ be a quad (i.e. a region of the plane homeomorphic to a disk, with two distinguished disjoint segments on its boundary). Then a.s. either all or none of the following events hold: (a) there is a continuous path included in $\calD_p \cap \calQ$ which joins one distinguished side of $\calQ$ to the other, (b) there is such a continuous path in $f^{-1}(]-p,+\infty[)$, (c) there is no continuous path included in $f^{-1}(]-\infty,-p]) \cap \calQ$ which joins one non-distinguished side of $\calQ$ to the other and (d) there is no such path in $f^{-1}(]-\infty,-p[)$. Similarly, if $\calA$ is an annulus, then a.s. either all of none of the following events hold: (a) there is a continuous path included in $\calD_p \cap \calA$ which separates the inner boundary of $\calA$ from its outer boundary, (b) there is such a path in $f^{-1}(]-p,+\infty[)$, (c) there is no continuous path in $f^{-1}(]-\infty,-p]) \cap \calA$ which joins the inner boundary of $\calA$ to its outer boundary and (d) there is no such path in $f^{-1}(]-\infty,-p[)$.\\
A consequence of these properties and of the fact that $f$ is centered is that, if we assume furthermore that $f$ is invariant by $\frac{\pi}{2}$-rotation, then the probability that there is a left-right crossing at level $0$ of the square $[0,s]^2$ is $1/2$ for any $s \in ]0,+\infty[$.
\end{remark}

The following lemma is a consequence of Lemma~\ref{l.transversality.1} and of Theorem~\ref{t.Pitt} and is crucial in the proof of box-crossing results.

\begin{lemma}[FKG]\label{l.FKG}
Let $f$ be a Gaussian field on $\R^2$ satisfying Condition \ref{a.super-std} such that $\kappa$ is $C^6$. Let $p \in \R$. Assume that for each $x\in\R^2$, $\kappa(x)\geq 0$. Let $A,B$ be  obtained by taking as unions and intersections of a finite number of crossings of quads and circuits in annuli above level $-p$. Then,
\[
\prob[A\cap B]\geq \prob[A]\prob[B]\, .
\]
\end{lemma}
\begin{proof}
It suffices to approximate the events by increasing events that depend on $f$ restricted a finitely many points and using Theorem~\ref{t.Pitt}. This can easily be done by considering the discrete model introduced in Section~\ref{s.weak-qi} and by using Lemma~\ref{l.transversality.1} to prove that the discrete crossing events indeed approximate the continuous crossing events (for a similar argument, see the proof of Theorem~\ref{t.qi_rect} in Subsection~\ref{ss.qi_cross}).
\end{proof}

The following lemma is useful to show that certain discrete approximations of events do converge a.s. to continuous geometric events. In the lemma we refer to the face-centered square lattice defined before (see Figure~\ref{f.face-centered}). We use this lemma only to study nodal components (see Subsection~\ref{ss.components}), but we do not need it in order to study crossing events.

\begin{lemma}\label{l.transversality.3}
Let $\calC\subseteq\R^2\setminus\{0\}$ be a compact smooth one-dimensional submanifold of $\R^2$ that intersects the axes $\{0\}\times\R$ and $\R\times\{0\}$ transversally. Assume that there is a finite number of $x\in\calC$ such that $T_x\calC$ is colinear to an edge of the face-centered square lattice. Then, for a.e. small enough $\eps>0$, we have:
\begin{enumerate}
\item the set $\calC$ does not intersect the vertex set and intersects each edge of $\calT^\eps$ transversally;
\item each edge of $\calT^\eps$ is intersected at most twice and any two distinct intersection points of $e$ are connected by a path in $\calC$ inside the union of the two faces adjacent to $e$.
\end{enumerate}
\end{lemma}
\begin{proof}
By simple application of Sard's theorem, the first property holds for a.e. $\eps>0$. We now take $\eps>0$ such that the first property holds and prove that the second property holds for $\eps>0$ small enough.  We begin by defining some constants depending on $\calC$ that will determine how small the $\eps$'s need to be to satisfy the second property.
\begin{itemize}
\item Since there are a finite number of points $x\in\calC$ such that $T_x\calC$ is colinear to an edge of the lattice, there exists $c_1>0$ such that any two such distinct points are at distance greater than $4c_1$.
\item The distance between any two distinct components of $\calC$ is bounded from below by a constant  $c_2>0$. 
\item Each component of $\calC$ is the image of some smooth embedding $\gamma:S^1\rightarrow\R^2$ with unit speed such that for each distinct $s,t\in S^1$, $|\gamma(s)-\gamma(t)|\geq \lambda_0\textup{dist}_{S^1}(s,t)$ (here and below, $\textup{dist}_{S^1}$ denotes the distance function on $S^1$). Let $\|\mathbf{k}\|_\infty<+\infty$ be the maximum of the curvature $|\mathbf{k}|$ on $\calC$ and let $c'>0$ be such that for any two points $x,y$ on a common edge $e$ and any point $z$ outside of the union of the two faces adjacent to $e$, the unit vectors $v_1$ and $v_2$ pointing in the directions of $z-x$ and $y-z$ satisfy $|v_1-v_2|\geq c'$. Let $c_3=c'\lambda_0/\|\mathbf{k}\|_\infty\in ]0,+\infty]$.
\end{itemize}
We take $\eps<\min(c_1,c_2,c_3)$ and prove that the second property holds.\\
Fix $e$ an edge of $\calT^\eps$. Let us prove that any two intersection points on $e$ must be connected by a smooth arc inside the union $F$ of the two faces adjacent to $e$. If $e$ is intersected at least twice, say at $x,y\in e$, then, $x,y$ are at distance less than $c_2$ so they must belong to the same component $\calC$. This component is parametrized by a smooth embedding $\gamma : S^1\rightarrow \R^2$ with unit speed so there are $s,t\in S^1$, such that $\gamma(s)=x$ and $\gamma(t)=s$. By assumption, $\eps\geq |x-y|\geq \lambda_0\textup{dist}_{S^1}(s,t)$. Assume that $x$ and $y$ are not connected by $\gamma$ inside the union $F$ of the two faces adjacent to $e$. Then, there exists $r\in S^1$ belonging to one of the shortest paths between $z$ and $t$ in $S^1$ such that $\gamma(r)=z\notin F$. We denote by $]s,t[$ the open interval in $S^1$ containing $r$, and denote by $]s,r[$ and $]r,t[$ the open sub-intervals with extremities $s$ and $r$ and $r$ and $t$ respectively. Let $v_1$ and $v_2$ be the unit tangent vectors pointing in the same directions as $z-x$ and $y-x$ respectively. By construction, $|v_1-v_2|\geq  c'$. By Rolle's theorem, there exist $u_1\in]s,r[$ and $u_2\in]r,t[$ such that $\gamma'(u_1)=v_1$ and $\gamma'(u_2)=v_2$. Moreover, by assumption, $\textup{dist}_{S^1}(u_1,u_2)\leq \lambda_0^{-1}\eps$. But this means that there exists $u_3\in]u_1,u_2[$ such that
\[
\|\mathbf{k}\|_\infty\geq |\mathbf{k}(\gamma(u_3))|=|\gamma''(u_3)|\geq \lambda_0c'\eps^{-1}\, .
\]
Consequently, $\eps\geq  \lambda_0c'/\|\mathbf{k}\|_\infty=c_3$ which contradicts our assumption. Therefore, $x$ and $y$ must be connected by a smooth arc.\\
Now, by Rolle's theorem, for any two distinct intersection points of $e$ connected by a smooth arc inside $F$, there must be a point $x$ on this connecting arc such that $T_x\calC$ is colinear to $e$. Thus, if $e$ contains three distinct intersection points, then there are two distinct points $x,y\in F$ such that $T_x\calC$ and $T_y\calC$ are colinear to $e$. But $x,y\in F$ so they must be at distance at most $4\eps\leq 4c_1$ which contradicts the definition of $c_1$. Hence, $|\calC\cap e|\leq 2$ and we are done.
\end{proof}

\section{A uniform discrete RSW estimate}\label{s.eps.RSW}

In this section, we prove a RSW result for the discrete models studied in~\cite{bg_16}. As explained in Section~\ref{s.intro}, contrary to~\cite{bg_16}, we do not use any discrete RSW estimate to deduce the continuous RSW estimate. However, a discrete RSW estimate uniform in the mesh $\eps$ can be useful if one wants to apply tools from discrete percolation to our model. The results of this section rely heavily on \cite{bg_16}. We also make a small correction in the arguments made therein. For these reasons, this Appendix should be read as a companion text to~\cite{bg_16}. We would like to stress the fact that \textbf{the results presented here are not used in the rest of the paper}. We first introduce the following notations:
\begin{notation}
Consider the discretized model introduced in the beginning of Section~\ref{s.weak-qi} and remember Definition~\ref{d.crosseps}. If $\calQ$ is a quad, write $\cross^\eps_0(\calQ)$ for the event that $\calQ$ is $\eps$-crossed at level $0$.
\end{notation}

We have the following result.

\begin{prop}\label{p.eps_rsw}
Let $f$ be a Gaussian field satisfying Conditions~\ref{a.super-std},~\ref{a.std},~\ref{a.decay2} as well as Condition~\ref{a.pol_decay} for some $\alpha > 4$. For every quad $\calQ$, there exist $s_0=s_0(\kappa,\calQ) \in ]0,+\infty[$ and $c=c(\kappa,\calQ)>0$ such that for each $s \in [s_0,+\infty[$ and each $\eps \in ]0,1]$ we have:
\[
\Pro \left[ \cross_0^\eps(s\calQ) \right] \geq c \, . 
\] 
\end{prop}
Note that the constant $c$ above does not depend on $\eps$. As in the continuous case, the first result of this kind can be found in~\cite{bg_16} by combining Theorem~2.2 of~\cite{bg_16} with their Section~4. The novelty here is that the result holds for any $\alpha > 4$ and without any constraint on $(s,\eps)$. As in the proof of Theorem~\ref{t.RSW}, we need a quasi-independence result to prove Proposition~\ref{p.eps_rsw}. We are going to use Proposition~\ref{p.eps_qi} where the quasi-independence estimate is uniform in $\eps$.

\begin{proof}[Proof of Proposition~\ref{p.eps_rsw}]
As in Section~\ref{s.Tassion}, we follow Tassion's strategy from~\cite{tassion2014crossing}. However, since we need a constant $c$ which is uniform in $\eps$, it is more suitable to follow the quantitative version of Tassion's method presented in Section~2 of~\cite{bg_16}.

Before going into the proof, let us warn the reader that in Section~\ref{s.Tassion} we have used the notations from~\cite{tassion2014crossing} while in the present appendix we use the notations from~\cite{bg_16}. In particular, the notation $\phi_s$ has two different meanings; we hope that this will not confuse the reader.

We first assume that $\eps^{-1} \in \Z_{> 0}$ so that our model is $\Z^2$-periodic. As noted in~\cite{bg_16}, by a simple duality argument (which works since our lattice is a triangulation), we obtain that the probability that there is a left-right crossing of $[-s/2,s/2]^2$ made of black edges of $\calT^\eps$ is $1/2$ for any $s \in 2\Z_{>0}$. Hence we have the existence of some $c_0 \in ]0,1[$ such that the probability of this event is at least $c_0$ for any $s\in 2\Z_{>0}$ as assumed in Condition~3 of Definition~2.1 in~\cite{bg_16}.
We first prove the following lemma analogous to Lemma~2.7 of~\cite{bg_16}. Our way to state this lemma is a little different from~\cite{bg_16} since we think that, for the proof of this lemma to be correct, one has to consider variants of the event $\calH_s(\cdot,\cdot)$ as we do below. The reason why we need to make such a change is that the models are not continuous, which implies that the function $\psi_s$ (which is defined in the proof) is not continuous, so the proof written in~\cite{tassion2014crossing} does not work as is. Let us stress that, once one has made this small correction, all the other results of~\cite{bg_16} hold without any modification.
\begin{lem}\label{l.P1andP2}
For any $s \geq 1$, $-s/2 \leq \alpha \leq \beta \leq s/2$, let $\calH_s(\alpha,\beta)$ (resp. $\widetilde{\calH}^1(\alpha,\beta)$, $\widetilde{\calH}^2(\alpha,\beta)$) be the event that there is a path in $[-s/2,s/2]^2$ from the left side to $\{s/2\} \times [\alpha,\beta]$ (resp. to $\{s/2\} \times ]\alpha,\beta]$, to $\{s/2\} \times [\alpha,\beta[$) made of black edges of $\calT^\eps$. Also, let $\calX_s(\alpha)$ be defined exactly as in~\cite{tassion2014crossing,bg_16} (see for instance Figure~2.2 of~\cite{bg_16}). There exists a universal polynomial $Q_1 \in \R[X]$, positive on $]0,1[$, such that for every $s \in 2\Z_{>0}$, there exists $\alpha_s=\alpha_s(\eps,\kappa) \in [0,s/4]$ satisfying the following properties:
\bi 
\item[\textup{(P1)}] $\Pro \left[ \calX_s(\alpha_s) \right] \geq Q_1(c_0) \, .$
\item[\textup{(P2)}] If $\alpha_s < s/4$, then $\Pro \left[ \calH_s(0,\alpha_s) \right] \geq c_0/4 + \Pro \left[ \widetilde{\calH}_s^1(\alpha_s,s/2) \right] \, .$ 
\ei
\end{lem} 
\begin{proof}
For every $\alpha \in [0,s/2]$, write:
\[
\psi_s(\alpha)=\psi_s(\kappa,\eps,\alpha)=\Pro \left[ \calH_s(0,\alpha) \right] - \Pro \left[ \calH_s(\alpha,s/2) \right] \, ,
\]
\[
\widetilde{\psi}_s^1(\alpha)=\widetilde{\psi}_s^1(\kappa,\eps,\alpha)=\Pro \left[ \calH_s(0,\alpha) \right] - \Pro \left[ \widetilde{\calH}_s^1(\alpha,s/2) \right]\, ,
\]
and
\[
\widetilde{\psi}_s^2(\alpha)=\widetilde{\psi}_s^2(\kappa,\eps,\alpha)=\Pro \left[ \widetilde{\calH}^2_s(0,\alpha) \right] - \Pro \left[ \calH_s(\alpha,s/2) \right]  \, .
\]
Note that:
\[
\forall \alpha \in [0,s/2[, \, \underset{\alpha' \rightarrow \alpha, \atop \alpha' > \alpha}{\lim} \psi_s(\alpha') = \widetilde{\psi}_s^1(\alpha) \, ; \, \forall \alpha \in ]0,s/2], \, \underset{\alpha' \rightarrow \alpha, \atop \alpha' < \alpha}{\lim} \psi_s(\alpha') = \widetilde{\psi}_s^2(\alpha) \, .
\]
Now, if $\Psi_s(s/4) > c_0/4$, then let $\alpha_s$ be the infimum over every $\alpha \in [0,s/4]$ such that $\psi_s(\alpha) > c_0/4$; otherwise let $\alpha_s=s/4$. Then, we have $\widetilde{\psi}_s^2(\alpha_s) \leq c_0/4$ and, if $\alpha_s < s/4$, we have $\widetilde{\psi}_s^1(\alpha_s) \geq c_0/4$. Thus, (P2) is satisfied. Concerning (P1), similarly as in Lemma~2.1 of~\cite{tassion2014crossing} we have:
\begin{eqnarray*}
c_0 & \leq & 2 \Pro \left[ \widetilde{\calH}_s^2(0,\alpha_s) \right] + 2 \Pro \left[ \calH_s(\alpha_s,s/2) \right]\\
& \leq & 4 \Pro \left[ \calH_s(\alpha_s,s/2) \right] + 2 \widetilde{\psi}^2_s(\alpha_s)\\
& \leq & 4 \Pro \left[ \calH_s(\alpha_s,s/2) \right] + c_0/2 \, .
\end{eqnarray*}
Finally, $\Pro \left[ \calH_s(\alpha_s,s/2) \right] \geq c_0/8$ thus as noted in~\cite{tassion2014crossing}, by a simple construction and by the FKG inequality we obtain that $\Pro \left[ \calX(\alpha_s) \right] \geq c_0\times (c_0/8)^4$ \, .
\end{proof}
Next, Lemmas~2.8 and~2.9 of~\cite{bg_16} apply readily. Now, define the universal funciton $\tau_1$ as in~(2.5) of~\cite{bg_16} and define the following function:
\[
\phi_s=\phi_s(\kappa,\eps) = \sup \big| \Pro \left[ A \cap B \right] - \Pro \left[ A \right] \Pro \left[ B \right] \big| \, ,
\]
where the supremum is over any event $A$ of the form $\Circ^\eps(\calA)$ where $\calA$ is an $\eps$-drawn annulus centered at $0$ and included in $[-s,s]^2$, and any event $B$ which is the intersection of at most $\log(s)$ events of the form $\Circ^\eps(\calA)$ where $\calA$ is an $\eps$-drawn annulus centered at $0$ and included in $[-s\log(s),s\log(s)]^2 \setminus ]-5s,5s[^2$. Next, write:
\[
\hat{s} = \hat{s}(\kappa,\eps) = \max \{ s \in \Z_{>0} \, : \, s \geq \exp(\tau_1(c_0)) \text{ and } \phi_s \geq \frac{c_0}{16} Q_3(c_0) \} \, ,
\]
where $Q_3$ is the universal positive function that comes from Lemma~2.9 of~\cite{bg_16}. We have the following lemma analogous to Lemma~2.10 of~\cite{bg_16}, where for any $0 < r < s <+\infty$, $\Circ_0^\eps(r,s)$ denotes the event that there is an $\eps$-circuit at level $0$ in the annulus $[-r,r]^2 \setminus ]-s,s[^2$, and where $Q_2$ is the universal positive function defined as in Lemma~2.8 of~\cite{bg_16}.
\begin{lem}
For any $s \in \Z_{>0}$, $s \geq \hat{s}$, if $\Pro \left[ \Circ^\eps_0(s,2s) \right] \geq Q_2(c_0)$, then there exists $s' \in [4s,\tau_1(c_0)s] \cap \Z_{>0}$ such that $\alpha_{s'} \geq s$.
\end{lem}
\begin{proof}
As noted in~\cite{bg_16}, since the rest of the proof is exactly the same as in~\cite{tassion2014crossing}, it is sufficient to prove that, if $s \geq \hat{s}$, then:
\begin{equation}\label{e.alpha_in_discrete_rsw}
\Pro \left[ \bigcap_{i=1}^{\lfloor \log_5(\tau_1) \rfloor} \Circ^\eps_0(5^is,2 \cdot 5^is)^c \right] < c_0/4 \, .
\end{equation}
The proof is the same as in~\cite{bg_16} since by our definition of $\hat{s}$, if $s \geq \hat{s}$ and if $i_0 \in \{ 1, \cdots, \lfloor \log_5(\tau_1) \rfloor - 1\}$, we have:
\begin{multline*}
\Pro \left[ \bigcap_{i=i_0}^{\lfloor \log_5(\tau_1) \rfloor} \Circ^\eps_0(\calA_{5^is,2 \cdot 5^is}0)^c \right]\\
\leq \Pro \left[ \Circ^\eps_0(\calA_{5^{i_0}s,2 \cdot 5^{i_0}s})^c \right]  \Pro \left[ \bigcap_{i=i_0+1}^{\lfloor \log_5(\tau_1) \rfloor} \Circ^\eps_0(\calA_{5^is,2 \cdot 5^is})^c \right]+ \frac{c_0}{16} Q_3 \, .
\end{multline*}
Note that here the fact that (P2) in Lemma~\ref{l.P1andP2} is written with $\widetilde{\calH}_s^1(\alpha_s,s/2)$ instead of $\calH_s(\alpha_s,s/2)$ does not change the proof at all.
\end{proof}
Now, define $\gamma(\nu)$, $t_\nu=t_\nu(\kappa,\eps)$ and $s_\nu=s_\nu(\kappa,\eps)$ as in~(2.8),~(2.9) and~(2.10) of~\cite{bg_16} with $\hat{s}$ instead of $s(\Omega)$ i.e.:
\begin{align*}
& \gamma(\nu) = 1+\log_{4/(3+2\nu)}(3/2+\nu) > 1 \, ,\\
& s_\nu=\max(\hat{s},\lfloor 6/\nu \rfloor + 1) \, ,\\
& t_\nu = (3/2+\nu) s_\nu^{\gamma(\nu)}\alpha_{s_\nu}^{1-\gamma(\nu)} \, .
\end{align*}
Then, the proof of Lemma~2.11 of~\cite{bg_16} applies readily with our definitions. Finally, as in the proof of Theorem~2.2 of~\cite{bg_16}, we obtain that for every $\nu \in ]0,1/2[$, there exists a universal positive continuous function $P_\nu$ defined on $[1,+\infty[ \times ]0,1[$ such that, for every $\rho \geq 1$ and every $s \in \Z_{>0}$ such that $s \geq t_\nu$, the probability that there is a black path in $[0,\rho s] \times [0,s]$ from the left side to the right side is at least $P_\nu(\rho,c_0)$.\\
At this point, we want to have an upper bound on $t_\nu=t_\nu(\eps)$ independent on $\eps$, i.e. we want to have an upper bound $\hat{s}=\hat{s}(\eps)$ and a lower bound on $\alpha_{s_\nu(\eps)}(\eps)$ that do not depend on $\eps$. To this purpose, first note that the functions $Q_2$, $Q_3$ and $P_\nu$ are continuous functions of $Q_1$ and that, as explained in Lemma~4.6 of~\cite{bg_16}, there exists $a=a(\kappa) > 0$ and $b=b(\kappa)>0$ such that, if one replace the universal function $Q_1$ by the function $aQ_1$ that depends only on $\kappa$, then we have $\alpha_s=\alpha_s(\kappa,\eps) \geq b$ for every $s$. More precisely, we can choose any $a \in ]0,1[$ and $b \in ]0,1/2[$ so that, for every $s$, the probability that $f$ is positive both in the $4b \times 4b$ box centered at $(-s/2,0)$ and the $4b \times 4b$ box centered at $(s/2,0)$ is at least $a$. Such quantities exist since $f$ is a.s. continuous and thanks to FKG. Secondly, note that, by Proposition~\ref{p.eps_qi}, $\phi_s$ is at most:
\[
C \, \log(s) \, (\log(s)s)^2 s^2 s^{-\alpha} \, .
\]
for some $C=C(\kappa)<+\infty$. Hence (and since $\alpha > 4$) $\hat{s}$ is less than some finite constant $M=M(\kappa)$ does not depend on $\eps$. Finally, $t_\nu$ is less than some finite constant that does not depend on $\eps$, and we have obtained Proposition~\ref{p.eps_rsw} for $\eps^{-1} \in \Z_{>0}$ and when the quad is a rectangle $[0,\rho] \times [0,1]$.\\
To end the proof, first note one can easily extend the result to any quad by reasoning as in the proof of Theorem~\ref{t.RSW}. Finally, to extend the result to any $\eps \in ]0,1]$, fix such an $\eps$, let $\lambda \in [1/2,2]$ such that $(\lambda \eps)^{-1} \in \Z_{>0}$ and define the planar Gaussian field $f_\lambda \, : \, x \mapsto f(\lambda x)$ with covariance function $(x,y) \mapsto \kappa_\lambda(x-y)$. For any $\eps' > 0$ and any quad $\calQ$, write $\cross_0^{\eps',\lambda}(\calQ)$ for the event $\cross_0^{\eps'}(\calQ)$ but with $f_\lambda$ instead of $f$. Note that we have:
\[
\cross^\eps_0(\calQ) = \cross^{\lambda\eps,\lambda}_0(\calQ) \, .
\]
Moreover, it is not difficult to see that, since $\lambda$ belongs to the compact subset of $]0,+\infty[$, $[1/2,2]$, one can find constant $a=a(\kappa_\lambda)$, $b=b(\kappa_\lambda)$ and $M=M(\kappa_\lambda,c_0)$ as above that are uniform in $\lambda$. This ends the proof.
\end{proof}

As in the continuous case, we can deduce that the one-arm event decreases polynomially fast. We first need a notation.
\begin{notation}
If $0<r<s<+\infty$, we write $\calA(r,s) = [-s,s]^2 \setminus ]-r,r[^2$ and we write $\arm_0^\eps(r,s$ (resp. $\arm_0^{*,\eps}(r,s)$) for the event that there is an $\eps$-black path rom the inner boundary of $\calA(r,s)$ to its outer boundary made of black edges (resp. that lives in the white region of the plane) in the discrete percolation model of mesh $\eps$ defined in the beginning of Section~\ref{s.weak-qi} with $p=0$.
\end{notation}

\begin{prop}\label{p.arm_eps}
Assume that $f$ satisfies Conditions~\ref{a.super-std},~\ref{a.std},~\ref{a.decay2} as well as Condition~\ref{a.pol_decay} for some $\alpha > 4$. There exists $C=C(\kappa)<+\infty$ and $\eta= \eta(\kappa) > 0$ such that, for each $\eps \in ]0,1]$, for each $s \in [1,+\infty[$ and $r \in [1,s[$:
\[
\Pro \left[ \arm_0^\eps(r,s) \right], \, \Pro \left[ \arm_0^{*,\eps}(r,s) \right] \leq C \, (r/s)^\eta \, .
\]
\end{prop}
\begin{proof}
First note that, since $f$ and $-f$ have the same law, we have:\footnote{These inequalities are not equalities only because black and white regions of the plane are not totally dual. These would be equalities if we had $\eps^{-1}r \in \N$ and $\eps^{-1}s \in \N$.}
 \[
 (\Pro \left[ \arm_0^{\eps,*}(r+\eps,s-\eps) \right] \leq) \, \Pro \left[ \arm_0^{\eps}(r,s) \right] \leq \Pro \left[ \arm_0^{\eps,*}(r,s) \right] \, .
 \]
So it is sufficient to prove the result for $\arm_0^{*,\eps}(r,s)$. The proof is roughly the same as the proof of Proposition~\ref{p.arm} except that we use Propositions~\ref{p.eps_rsw} and~\ref{p.eps_qi} instead of Theorem~\ref{t.RSW} and Theorem~\ref{t.qi_rect}. The only difference is that we have to consider only $\eps$-annuli, but that is not a problem. The constants do not depend on $\eps$ since the constants in Propositions~\ref{p.eps_rsw} and~\ref{p.eps_qi} do not. 
\end{proof}

As in the continuous case, the following is a direct consequence of Proposition~\ref{p.arm_eps}:
\begin{prop}
With the same hypotheses as Proposition~\ref{p.arm_eps}, for each $\eps \in ]0,1]$ a.s. there is no unbounded black component in the discrete percolation model of mesh $\eps$ defined in the beginning of Section~\ref{s.weak-qi} with $p=0$.
\end{prop}

\bibliographystyle{alpha}
\bibliography{ref_perco}

\begin{thebibliography}{DCHN11}

\bibitem[{Ale}96]{alex_96}
Kenneth~S. {Alexander}.
\newblock {Boundedness of level lines for two-dimensional random fields.}
\newblock {\em The Annals of Probability}, 24(4):1653--1674, 1996.

\bibitem[AT07]{adta_rfg}
Robert~J. Adler and Jonathan~E. Taylor.
\newblock {\em Random fields and geometry}.
\newblock Springer, 2007.

\bibitem[ATT16]{ahlberg2016sharpness}
Daniel Ahlberg, Vincent Tassion, and Augusto Teixeira.
\newblock Sharpness of the phase transition for continuum percolation in
  {$\R^2$}.
\newblock {\em arXiv preprint arXiv:1605.05926}, 2016.

\bibitem[AW09]{azws}
Jean-Marc Aza{\"{\i}}s and Mario Wschebor.
\newblock {\em Level sets and extrema of random processes and fields}.
\newblock John Wiley \& Sons, Inc., Hoboken, NJ, 2009.

\bibitem[BG16]{bg_16}
Vincent Beffara and Damien Gayet.
\newblock Percolation of random nodal lines.
\newblock {\em arXiv preprint arXiv:1605.08605, to appear in Inst. Hautes
  Etudes Sci. Publ. Math.}, 2016.

\bibitem[BG17]{bg_17}
Vincent Beffara and Damien Gayet.
\newblock Percolation without {F}{K}{G}.
\newblock {\em arXiv preprint arXiv:1710.10644}, 2017.

\bibitem[BM18]{bm_17}
Dmitry Beliaev and Stephen Muirhead.
\newblock Discretisation schemes for level sets of planar gaussian fields.
\newblock {\em Communications in Mathematical Physics}, 359(3):869--913, 2018.

\bibitem[BMW17]{bmw_17}
Dmitry Beliaev, Stephen Muirhead, and Igor Wigman.
\newblock Russo-{S}eymour-{W}elsh estimates for the {K}ostlan ensemble of
  random polynomials.
\newblock {\em arXiv preprint arXiv:1709.08961}, 2017.

\bibitem[BR06a]{bollobas2006critical}
B{\'e}la Bollob{\'a}s and Oliver Riordan.
\newblock The critical probability for random {V}oronoi percolation in the
  plane is {$1/2$}.
\newblock {\em Probability theory and related fields}, 136(3):417--468, 2006.

\bibitem[BR06b]{bollobas2006percolation}
B{\'e}la Bollob{\'a}s and Oliver Riordan.
\newblock {\em Percolation}.
\newblock Cambridge University Press, 2006.

\bibitem[BS07]{bs_07}
Eugene Bogomolny and Charles Schmit.
\newblock Random wavefunctions and percolation.
\newblock {\em Journal of Physics A: Mathematical and Theoretical},
  40(47):14033, 2007.

\bibitem[BS15]{basu2015crossing}
Deepan Basu and Artem Sapozhnikov.
\newblock Crossing probabilities for critical {B}ernoulli percolation on slabs.
\newblock {\em arXiv preprint arXiv:1512.05178}, 2015.

\bibitem[CL09]{cheney2009course}
Elliott~Ward Cheney and William~Allan Light.
\newblock {\em A course in approximation theory}, volume 101.
\newblock American Mathematical Soc., 2009.

\bibitem[DCHN11]{duminil2011connection}
Hugo Duminil-Copin, Cl{\'e}ment Hongler, and Pierre Nolin.
\newblock Connection probabilities and {RSW}-type bounds for the
  two-dimensional {FK} {I}sing model.
\newblock {\em Communications on Pure and Applied Mathematics},
  64(9):1165--1198, 2011.

\bibitem[DCTT16]{duminil2016box}
Hugo Duminil-Copin, Vincent Tassion, and Augusto Teixeira.
\newblock The box-crossing property for critical two-dimensional oriented
  percolation.
\newblock {\em arXiv preprint arXiv:1610.10018}, 2016.

\bibitem[Gri99]{grimmett1999percolation}
Geoffrey~R. Grimmett.
\newblock {\em Percolation (Grundlehren der mathematischen Wissenschaften)}.
\newblock Springer: Berlin, Germany, 1999.

\bibitem[Gri10]{grimmett2010probability}
Geoffrey Grimmett.
\newblock Probability on graphs.
\newblock {\em Lecture Notes on Stochastic Processes on Graphs and Lattices.
  Statistical Laboratory, University of Cambridge}, 2010.

\bibitem[GW11]{gw_10}
Damien {Gayet} and Jean-Yves {Welschinger}.
\newblock {Exponential rarefaction of real curves with many components.}
\newblock {\em { Inst. Hautes \'Etudes Sci. Publ. Math.}}, 113:69--96, 2011.

\bibitem[Har60]{harris1960lower}
Theodore~E. Harris.
\newblock A lower bound for the critical probability in a certain percolation
  process.
\newblock In {\em Mathematical Proceedings of the Cambridge Philosophical
  Society}, volume~56, pages 13--20. Cambridge Univ Press, 1960.

\bibitem[Kes80]{kesten1980critical}
Harry Kesten.
\newblock The critical probability of bond percolation on the square lattice
  equals {$1/2$}.
\newblock {\em Comm. Math. Phys. 74, no. 1}, 1980.

\bibitem[Let18]{let_16}
Thomas Letendre.
\newblock Variance of the volume of random real algebraic submanifolds.
\newblock {\em Transactions of the American Mathematical Society}, 2018.

\bibitem[MS83a]{molchanov1983percolationi}
Stanislav~A. Molchanov and A.K. Stepanov.
\newblock Percolation in random fields. i.
\newblock {\em Theoretical and Mathematical Physics}, 55(2):478--484, 1983.

\bibitem[MS83b]{molchanov1983percolationii}
Stanislav~A. Molchanov and A.K. Stepanov.
\newblock Percolation in random fields. ii.
\newblock {\em Theoretical and Mathematical Physics}, 55(3):592--599, 1983.

\bibitem[MS86]{molchanov1986percolationiii}
Stanislav~A. Molchanov and A.K. Stepanov.
\newblock Percolation in random fields. iii.
\newblock {\em Theoretical and Mathematical Physics}, 67(2):434--439, 1986.

\bibitem[NS09]{naso_nod}
Fedor Nazarov and Mikhail Sodin.
\newblock On the number of nodal domains of random spherical harmonics.
\newblock {\em Amer. J. Math.}, 131(5):1337--1357, 2009.

\bibitem[NS11]{ns_2010}
Fedor Nazarov and Mikhail Sodin.
\newblock Fluctuations in random complex zeroes: asymptotic normality
  revisited.
\newblock {\em Int. Math. Res. Not. IMRN}, (24):720--5759, 2011.

\bibitem[NS16]{nazarov2015asymptotic}
Fedor Nazarov and Mikhail Sodin.
\newblock Asymptotic laws for the spatial distribution and the number of
  connected components of zero sets of {G}aussian random functions.
\newblock {\em Zh. Mat. Fiz. Anal. Geom., 12(3):205–278}, 2016.

\bibitem[NSV07]{nsv_2007}
Fedor Nazarov, Mikhail Sodin, and Alexander Volberg.
\newblock Transportation to random zeroes by the gradient flow.
\newblock {\em Geometric and functional analysis}, 17(3):887--935, 2007.

\bibitem[NSV08]{nsv_2008}
Fedor Nazarov, Mikhail Sodin, and Alexander Volberg.
\newblock The {J}ancovici--{L}ebowitz--{M}anificat law for large fluctuations
  of random complex zeroes.
\newblock {\em Communications in mathematical physics}, 284(3):833--865, 2008.

\bibitem[NTW17]{newman2017critical}
Charles Newman, Vincent Tassion, and Wei Wu.
\newblock Critical percolation and the minimal spanning tree in slabs.
\newblock {\em Communications on Pure and Applied Mathematics},
  70(11):2084--2120, 2017.

\bibitem[Pit82]{pitt1982positively}
Loren~D. Pitt.
\newblock Positively correlated normal variables are associated.
\newblock {\em The Annals of Probability}, pages 496--499, 1982.

\bibitem[{Pit}96]{piterbarg}
V.I. {Piterbarg}.
\newblock {\em {Asymptotic methods in the theory of Gaussian processes and
  fields. Transl. from the Russian by V. V. Piterbarg. Transl. ed. by Simeon
  Ivanov.}}
\newblock Providence, RI: AMS, 1996.

\bibitem[Rus78]{russo1978note}
Lucio Russo.
\newblock A note on percolation.
\newblock {\em Probability Theory and Related Fields}, 43(1):39--48, 1978.

\bibitem[RV17]{rv_bf}
Alejandro Rivera and Hugo Vanneuville.
\newblock The critical threshold for {B}argmann-{F}ock percolation.
\newblock {\em arXiv preprint arXiv:1711.05012}, 2017.

\bibitem[Sle62]{slepian_1962}
David Slepian.
\newblock The one-sided barrier problem for {G}aussian noise.
\newblock {\em Bell Labs Technical Journal}, 41(2):463--501, 1962.

\bibitem[SW78]{seymour1978percolation}
Paul~D. Seymour and Dominic~J.A. Welsh.
\newblock Percolation probabilities on the square lattice.
\newblock {\em Annals of Discrete Mathematics}, 3:227--245, 1978.

\bibitem[Tas16]{tassion2014crossing}
Vincent Tassion.
\newblock Crossing probabilities for {V}oronoi percolation.
\newblock {\em The Annals of Probability}, 44(5):3385--3398, 2016.

\end{thebibliography}

\ \\
{\bf Alejandro Rivera}\\
Univ. Grenoble Alpes, UMR5582, Institut Fourier, 38000 Grenoble, France\\
alejandro.rivera@univ-grenoble-alpes.fr\\
Supported by the ERC grant Liko No 676999\\

\medskip
\ni
{\bf Hugo Vanneuville}\\
Univ. Lyon 1, UMR5208, Institut Camille Jordan, 69100 Villeurbanne, France\\
vanneuville@math.univ-lyon1.fr\\
\url{http://math.univ-lyon1.fr/~vanneuville/}\\
Supported by the ERC grant Liko No 676999\\

\end{document}